\documentclass[twoside,10pt]{article}

\usepackage{lipsum}
\usepackage[pagewise]{lineno}

\usepackage[colorlinks,
            citecolor=red]{hyperref}

\usepackage{makecell}
\usepackage{multirow}

\usepackage[english]{babel}
\usepackage{amsmath}
\usepackage{amsthm}
\usepackage{amsfonts}
\usepackage{amssymb}
\usepackage{graphicx}
\usepackage{colortbl,dcolumn}
\usepackage{paralist}  
\usepackage{latexsym}
\usepackage{amsfonts}
\usepackage{amssymb}
\usepackage{cite}
\usepackage{appendix}
\usepackage{subcaption}
\usepackage{tikz}

     \newtheorem{lemma}{\bf Lemma}[section]
     \newtheorem{theorem}{\bf Theorem}[section]

     \newtheorem{definition}{\bf Definition}[section]
     \newtheorem{remark}{\bf Remark}[section]

     \numberwithin{equation}{section}

\begin{document}
\title{{\LARGE Effective Boundary Conditions Arising from the Heat Equation with Three-dimensional Interior Inclusion}
 \footnotetext{
 E-mail addresses: gengxingri@u.nus.edu.\\ }}

\author{{Xingri Geng$^{a, b}$}\\[2mm]
\small\it $^a$ Department of Mathematics, Southern University of Science and Technology, Shenzhen, P.R. China\\
\small\it $^b$ Department of Mathematics, National University of Singapore,  Singapore 
}

\date{}

\maketitle
\begin{abstract}
We study the initial boundary value problem for a heat equation in a domain containing a thin layer. The thermal conductivity of the layer is drastically different from that of the bulk of the domain; moreover, the layer is anisotropic and ``optimally aligned" in the sense that the normal direction in the layer is always an eigenvector of the thermal tensor. To reveal the effects of the layer, we regard it as a thickless surface on which ``effective boundary conditions" (EBCs) are satisfied by the limit of solutions of the initial boundary value problem as the thickness of the layer shrinks to zero. These EBCs are rich in variety and type, including some nonstandard ones such as the Dirichlet-to-Neumann mapping and the fractional Laplacian.
\end{abstract}


\medskip


\section{Introduction}
We are concerned in this paper with a heat equation in a domain with an interior layer that is thin compared to the scale of the whole domain. Moreover, the thermal tensor inside the layer differs significantly from that outside. The multi-scales in the spatial size and the drastic difference in the thermal tensor inevitably lead to a computational burden. An example of this kind of situation is insulating an inner portion of a conducting body with a thin layer (see Figure \ref{fig}). One natural way to deal with situations like this is to treat the thin layer as a thickless surface, on which we impose ``effective boundary conditions'' (EBCs). These EBCs are not only helpful in numerical computations but also give us a straightforward assessment of the effects of the layer. 

We aim to derive these EBCs rigorously. The case when the domain is two-dimensional was already treated by Li and Wang in \cite{LW}. In this paper, we treat the case when the domain is three-dimensional and the interior layer is ``optimally aligned", which was not considered in \cite{LW}. Compared to the two-dimensional case, the EBCs in the three-dimensional case are richer in variety; the method used is different.

To set the stage, let $\Omega_\delta$ be a thin layer inside the bounded domain $\Omega=\overline{\Omega}_1\cup\overline{\Omega}_\delta\cup\Omega_2\subset \mathbb{R}^3$, as illustrated in Figure \ref{fig}. For any finite $T>0$, consider the initial boundary value problem 
\begin{equation*} \label{PDE}
    \left\{
             \begin{array}{lll}
             u_t-\nabla \cdot (A(x)\nabla u)=f(x,t), &\mbox{$(x,t)\in Q_T,$} \\
             u=0,     &\mbox{$(x,t)\in S_T,$}       &\tag{1.1}\\
             u=u_0, &\mbox{$(x,t)\in\Omega\times\{0\},$}  
             \end{array}
  \right.
 \end{equation*}
where $Q_T=\Omega \times (0, T)$  and  $S_T=\partial\Omega\times(0, T)$; $u_0 \in L^2(\Omega)$, $f\in L^2(Q_T)$, and the thermal tensor $A(x)$ is given by
\begin{equation*}\label{A}
    A(x)=\left\{
        \begin{array}{ll}
              k_1 I_{3\times3}, &x\in \Omega_1,\\
             (a_{ij}(x))_{3\times3}, &x\in\Omega_\delta, \\ 
              k_2I_{3\times3}, & x\in\Omega_2.
        \end{array}
    \right.\tag{1.2}
\end{equation*}
The positive definite matrix $ \left(a_{ij}(x) \right)$ satisfies anisotropy and optimally aligned conditions in $\Omega_\delta$ in the following sense 
\begin{equation*}\label{tensor}
\begin{split}
    A(x)\textbf{n}(p)=\sigma \textbf{n}(p), \quad  A(x)\textbf{s}(p)=\mu \textbf{s}(p),
\end{split}\tag{1.3}
\end{equation*}
where $p$ is the projection of $x$ on $\Gamma_1(=\partial\Omega_1)$; $\textbf{n}(p)$
is the unit normal vector of $\Gamma_1$ pointing out of $\Omega_1$ at $p$, and $\textbf{s}(p)$ is an arbitrary tangent vector on $\Gamma_1$ at $p$;
$\sigma$ and $\mu$ are two corresponding eigenvalues. 

The notion of optimally aligned layers was first proposed by Rosencrans and Wang \cite{RW} in 2006, in which $\sigma$ is called normal conductivity, and $\mu$ is called tangent conductivity. It is clear that if $\sigma=\mu$, then $A(x)$ is also isotropic in $\Omega_\delta$.

Throughout the paper, we shall assume that $\sigma = \sigma(\delta)$ and $\mu = \mu(\delta)$ are two positive functions of $\delta$; $k_1$ and $k_2$ are two positive constants independent of $\delta> 0$; both $\Omega$ and $\Omega_1$ are fixed and bounded with $C^2$ smooth boundaries $\partial\Omega$ and $\Gamma_1(=\partial\Omega_1)$, respectively. Furthermore, the thin layer $\Omega_\delta$ is uniformly thick with thickness $\delta$, and $\Gamma_2$ converges to $\Gamma_1$ as $\delta \to 0$. 

\begin{center}\label{fig}
\begin{tikzpicture}
\def\angle{60}%
\pgfmathsetlengthmacro{\xoff}{2cm*cos(\angle)}%
\pgfmathsetlengthmacro{\yoff}{1cm*sin(\angle)}%
\draw (0,0) circle[x radius=3cm, y radius=2cm] ++(-1*\xoff,1.8*\yoff) node{$\Omega_2$};
\draw (0,0) circle[x radius=2.2cm, y radius=1.2cm]  ++(1.8*\xoff,1*\yoff)  node{$\Gamma_2$}  ++(0*\xoff,0.9*\yoff)  node{$\partial\Omega$}   ;
\draw (0,0) circle[x radius=2cm, y radius=1cm] node{$\Omega_1$}
++(0*\xoff,-2.8*\yoff) node{Figure 1: $\Omega=\overline{\Omega}_1\cup\overline{\Omega}_\delta\cup\Omega_2.$} ++(1.8*\xoff,2.6*\yoff)  node{$\Gamma_1$};
\draw ++(-3*\xoff,1.2*\yoff)node{$\Omega_\delta$};
\draw [->] (-3,1.1)--(-2.1,0.1);

\fill (0,1) circle [radius = 1pt];
\draw ++(0*\xoff,0.8*\yoff) node{\textbf{p}};

\draw[->] (0,1)--(0,1.4);
\draw ++(0*\xoff,1.8*\yoff) node{\textbf{n}};
\draw[->] (0,1)--(0.4,1);
\draw ++(0.6*\xoff,1*\yoff) node{\textbf{s}};
\end{tikzpicture}
\end{center}

We now introduce several essential Sobolev spaces. Let $W^{1,0}_2(Q_T)$ be the subspace of functions of $L^2(Q_T)$ with first order weak derivatives in $x$ also in $L^2(Q_T)$; $W^{1,1}_2(Q_T)$ is defined similarly with the first order weak derivative in $t$ also in $L^2(Q_T)$. Moreover, $W^{1,0}_{2,0}(Q_T)$ is the closure in $W^{1,0}_2(Q_T)$ of $C^\infty$ functions vanishing near $\overline {S}_T$, and $W^{1,1}_{2,0}(Q_T)$ is defined similarly. Finally, denote
\begin{equation*}
    V^{1,0}_{2,0}(Q_T) = W^{1,0}_{2,0}(Q_T)\cap C  \left([0, T];L^2(\Omega)\right).
\end{equation*}

Let us add two more Sobolev spaces that will appear in the future sections: let $Q_T^1=\Omega_1\times(0, T)$ and $Q_T^2=(\Omega\backslash\overline{\Omega}_1)\times(0, T)$; let
\begin{equation*}
    \begin{split}
         \widetilde{W}^{1,1}_{2,0}(Q_T) &= \left\{ u\in L^2(Q_T)  \big| \quad u_1 = u|_{Q_T^1} \in W^{1,0}_{2}(Q_T^1), \quad u_2 =  u|_{Q_T^2} \in W^{1,0}_{2}(Q_T^2), \quad u_2|_{S_T} = 0 \right\};\\
        \widetilde{V}^{1,0}_{2,0}(Q_T) & = \Big\{ u\in L^2(Q_T) \big| \quad u_1 \in W^{1,0}_{2}(Q_T^1), \quad u_2 \in W^{1,0}_{2}(Q_T^2), \quad u_2|_{S_T} = 0, \quad u\in C([0, T]; L^2(\Omega))\Big\}.
    \end{split}
\end{equation*}
We endow all these Sobolev spaces with natural norms.

\medskip

To simplify notations, we replace $\int_0^T\int_{\Omega} u(x, t) dxdt$ by $\int_{Q_T} u(x, t) dxdt$.
\begin{definition}\label{def11}
A function $u$ is said to be a weak solution of (\ref{PDE}), if $u\in V^{1,0}_{2,0}(Q_T)$ and for any $\xi\in W^{1,1}_{2,0}(Q_T) $ satisfying $\xi=0$ at $t=T$, it holds that
\begin{equation*}\label{weaksol}
   \mathcal{A}[u,\xi]:= -\int_{\Omega}u_0\xi(x,0)dx+\int_{Q_T} \left( A(x)\nabla u \cdot \nabla\xi   -u\xi_t-f\xi \right) dxdt=0.  \tag{1.4}
\end{equation*}
\end{definition}

For any small $\delta > 0$, (\ref{PDE}) admits a unique weak solution $u \in W^{1,0}_{2}(Q_T)\cap C  \left([0, T];L^2(\Omega)\right)$. It is well known that the following ``transmission conditions" are satisfied in the weak sense,
\begin{equation*}\label{trans}
    \left\{
             \begin{array}{lr}
             u_1=u_\delta, \quad k_1\nabla u_1 \cdot \textbf{n}_1=\sigma \nabla u_\delta \cdot \textbf{n}_1 \quad \text{ on } \quad \Gamma_1,&  \\
             u_2=u_\delta, \quad k_2\nabla u_2 \cdot \textbf{n}_2=\sigma \nabla u_\delta \cdot \textbf{n}_2 \quad \text{ on } \quad \Gamma_2,& \end{array}\tag{1.5}
 \right.
\end{equation*}
where $u_1$, $u_\delta$, and $u_2$ are the restrictions of $u$ on $\Omega_1\times (0, T)$, $\Omega_\delta\times (0, T)$, and $\Omega_2\times (0, T)$, respectively; $\textbf{n}_1$ and $\textbf{n}_2$ are the outward unit normal vector of $\Gamma_1$ and $ \Gamma_2$, respectively.

\medskip

Denote 
\begin{equation}
              A_0(x)=\left\{
               \begin{array}{ll}
                 k_1,&  x\in \Omega_1, \\ \notag
                 k_2,& x\in\Omega \backslash \overline{\Omega}_1.
               \end{array}
               \right.
\end{equation}
The main purpose of this work is to obtain effective boundary conditions on $\Gamma_1 \times (0, T)$ for the Dirichlet problem (\ref{PDE}) as the thickness of the layer decreases to zero.

\begin{theorem}\label{thm:bigthm}
Suppose that $A(x)$ is given by (\ref{A}) and (\ref{tensor}), $u_0 \in L^2(\Omega)$, and $f \in L^2(Q_T)$, assume further that $\sigma$ and $\mu$ satisfy the following scaling relationships
\begin{equation*}
\begin{split}
     \lim_{\delta\to 0}\frac{\sigma}{\delta} =b\in[0,\infty], \quad
 \lim_{\delta\to 0}\sigma\mu =\gamma\in[0,\infty], \quad
 \lim_{\delta\to 0}\mu\delta=\beta\in[0,\infty].
\end{split}
\end{equation*}
If $u$ is the weak solution of (\ref{PDE}), 
then as $\delta \to 0$, $u\to v$ strongly in $C([0,T];L^2(\Omega))$, where $v$ is the weak solution of
\begin{equation}\label{EPDE}
    \left\{
             \begin{array}{llr}
             v_t-\nabla \cdot (A_0(x)\nabla v)=f(x,t), & (x,t)\in Q_T,&  \\
             v=0, & (x,t)\in S_T,& \\ \tag{1.6}
             v=u_0, & x \in \Omega, t=0, &  
             \end{array}
\right.
\end{equation}
subject to the effective boundary conditions listed in Table \ref{tb1} with $v_1$ and $v_2$ being the restrictions of $v$ on $\Omega_1 \times (0, T)$ and $\left(\Omega \backslash \overline{\Omega}_1\right) \times (0, T)$, respectively. 
\end{theorem}

We now turn to the explanation of what we mean by the boundary conditions arising in Table \ref{tb1}. The dashed lines in Table \ref{tb1} indicate that such cases do not exist. 

The boundary condition $\nabla_{\Gamma_1} v =0$ on $\Gamma_1 \times (0, T)$ means that $v$ is a constant in the spatial variable, whereas it can be a function in the time variable, where $\nabla_{\Gamma_1}$ is the surface gradient on $\Gamma_1$. The boundary condition
\begin{equation*}
    \begin{split}
k_1\frac{\partial v_1}{\partial \textbf{n}}-k_2\frac{\partial v_2}{\partial \textbf{n}}=\beta \Delta_{\Gamma_1} v
    \end{split}
\end{equation*}
can be understood as a second-order partial differential equation on $\Gamma_1$, where $\Delta_{\Gamma_1}  = \nabla_{\Gamma_1} \cdot \nabla_{\Gamma_1}$ is the Laplace-Beltrami operator defined on $\Gamma_1$. This condition reveals that the thermal flux across $\Gamma_1$ in the outer normal direction is not equal to that in the inner normal direction, so the heat gets trapped and then diffuses on $\Gamma_1$ with diffusion rate $\beta$. 

$\mathcal{J}_1^{\beta/\gamma}$ and $\mathcal{J}_2^{\beta/\gamma}$ in Table \ref{tb1} are linear symmetric operators mapping the Dirichlet boundary value to the Neumann boundary value. In particular, for $H \in (0,\infty)$ and a smooth function $g$ on $\Gamma_1$, they are defined by 
\begin{equation*}
    \begin{split}
    \mathcal{J}_1^{H}[g](s) := \Psi_R(s, 0), \quad \mathcal{J}_2^{H}[g](s) := \Psi_R(s, H),
    \end{split}
\end{equation*}
where $\Psi(s, R)$ is the bounded solution of 
\begin{equation*}
    \left\{
             \begin{array}{ll}
                   \Psi_{RR}+\Delta_{\Gamma_1}      \Psi=0  & \Gamma_1\times(0, H),\\
                  \Psi (s, 0)=g(s) &      \Psi(s, H)= 0.
             \end{array}
  \right.
\end{equation*}
 
The analytic formulas for $\mathcal{J}_1^{H}[g]$ and $\mathcal{J}_2^{H}[g]$ are given in eigenfunctions of the Laplacian-Beltrami operator $-\Delta_{\Gamma_1}$ and thus deferred to Section \ref{sec 3}. With the help of these analytic formulas, we then define
\begin{equation*}
    \begin{split}
        \mathcal{J}_1^{\infty}[g](s) :=& \underset{H \to \infty}{\lim}\mathcal{J}_1^{H}[g](s), \quad \mathcal{J}_2^{\infty}[g](s) := \underset{H \to \infty}{\lim}\mathcal{J}_2^{H}[g](s).
    \end{split}
\end{equation*}

\begin{table}[!ht]
\caption{\label{tb1} Effective boundary conditions on $\Gamma_1\times (0, T)$ for the Dirichlet problem (\ref{PDE}).}

\centering Case 1. $ \frac{\sigma}{\delta} \to 0$ as $\delta \to 0$.

\medskip

\begin{tabular}{l|lll}
         
         As $\delta \to 0$ & \qquad $\gamma =0$ 
        &  $\gamma\in (0,\infty)$ 
        & $\gamma = \infty$\\
        \hline
        $\beta = 0$
        & \makecell {$ k_1\frac{\partial v_1}{\partial \textbf{n}}=0$, \\
        $ k_1\frac{\partial v_1}{\partial \textbf{n}}=k_2\frac{\partial v_2}{\partial \textbf{n}}$ \\
        }
        &  $------$ 
        &  $------$ \\
        \hline
        $\beta \in (0,\infty)$
      & \makecell {$ k_1\frac{\partial v_1}{\partial \textbf{n}}=0$, \\
        $ k_1\frac{\partial v_1}{\partial \textbf{n}}=k_2\frac{\partial v_2}{\partial \textbf{n}}$ }
      &   $------$
      &  $------$
      \\
      \hline
      $\beta = \infty$ 
      & $ k_1\frac{\partial v_1}{\partial \textbf{n}}=0$,
        $ k_1\frac{\partial v_1}{\partial \textbf{n}}=k_2\frac{\partial v_2}{\partial \textbf{n}}$ 
      & \makecell{ \quad$ k_1\frac{\partial v_1}{\partial \textbf{n}}=\gamma \mathcal{J}_1^{\infty} [v_1]$,\\ 
      $k_2\frac{\partial v_2}{\partial \textbf{n}}$
        $= -\gamma \mathcal{J}_1^{\infty} [v_2]$}
      &  \makecell{  $\nabla_{\Gamma_1} v_1 = \nabla_{\Gamma_1} v_2  =0$,\\
      $\int_{\Gamma_1} k_1\frac{\partial v_1}{\partial \textbf{n}}ds = 0$,\\
          $\int_{\Gamma_1} \left(k_1\frac{\partial v_1}{\partial \textbf{n}} - k_2\frac{\partial v_2}{\partial \textbf{n}}\right) ds = 0 $
           }\\
    \hline
\end{tabular}

\bigskip

\quad \qquad Case 2. $ \frac{\sigma}{\delta} \to b\in (0,\infty)$ as $\delta \to 0$.

\medskip

\begin{tabular}{l|lll}
        As $\delta \to 0$
        &  $\gamma =0$ 
        &  $\gamma\in (0,\infty)$ 
        & $\gamma = \infty$\\
        \hline
        $\beta = 0$
        & \makecell {$ k_1\frac{\partial v_1}{\partial \textbf{n}}=k_2\frac{\partial v_2}{\partial \textbf{n}}$,\\$b(v_2-v_1)=k_1\frac{\partial v_1}{\partial\textbf{n}}$ }
        &  $------$ 
        &  $------$ \\
        \hline
        $\beta \in (0,\infty)$
      &  $------$ 
      &  \makecell{
      $ k_1\frac{\partial v_1}{\partial \textbf{n}} =$\\
      $\gamma \mathcal{J}_1^{\beta /\gamma} [v_1]-\gamma \mathcal{J}_2^{\beta/\gamma} [v_2]$,\\
      $ k_2\frac{\partial v_2}{\partial \textbf{n}} =$\\
      $\gamma \mathcal{J}_2^{\beta /\gamma} [v_1]-\gamma \mathcal{J}_1^{\beta/\gamma} [v_2]$}
      &  $------$
      \\
      \hline
      $\beta = \infty$ 
      &  $------$
      & $------$
      & \makecell{$\nabla_{\Gamma_1} v_1 = \nabla_{\Gamma_1} v_2  =0$,\\
         $\int_{\Gamma_1} \left(k_1\frac{\partial v_1}{\partial \textbf{n}}-k_2\frac{\partial v_2}{\partial \textbf{n}}\right) ds=0 $,\\
         $\int_{\Gamma_1} \left( k_1\frac{\partial v_1}{\partial \textbf{n}} - b (v_2-v_1) \right) ds=0$} \\  
    \hline
\end{tabular}

\bigskip

\medskip

\qquad \qquad Case 3. $ \frac{\sigma}{\delta} \to \infty$ and $   \sigma \delta^3  \to 0$ as $\delta \to 0$. 

\medskip

\begin{tabular}{l|lll}
         As $\delta \to 0$
        &  $\gamma =0$ 
        &  $\gamma\in (0,\infty)$ 
        & $\gamma = \infty$\\
        \hline
        $\beta = 0$
        & \makecell {$v_1 = v_2$,
        $ k_1\frac{\partial v_1}{\partial \textbf{n}}=k_2\frac{\partial v_2}{\partial \textbf{n}}$ }
        &  \makecell {$v_1 = v_2$,
        $ k_1\frac{\partial v_1}{\partial \textbf{n}}=k_2\frac{\partial v_2}{\partial \textbf{n}}$ }
        & \makecell {$v_1 = v_2$,\\
        $ k_1\frac{\partial v_1}{\partial \textbf{n}}=k_2\frac{\partial v_2}{\partial \textbf{n}}$ }\\
        \hline
        $\beta \in (0,\infty)$
      &  $------$ 
      &  $------$ 
      & \makecell {$v_1 = v_2$,\\  $ k_1\frac{\partial v_1}{\partial \textbf{n}}-k_2\frac{\partial v_2}{\partial \textbf{n}}= \beta \Delta_{\Gamma_1} v $}
      \\
      \hline
      $\beta = \infty$ 
      &   $------$ 
      & $------$ 
      & \makecell{$ v_1 =  v_2  $,\\
            $\nabla_{\Gamma_1} v =0$,\\
         $\int_{\Gamma_1} \left(k_1\frac{\partial v_1}{\partial \textbf{n}}-k_2\frac{\partial v_2}{\partial \textbf{n}} \right) ds=0 $}   \\
    \hline
\end{tabular}
\end{table}

As we can see from Section \ref{sec 3}, $\mathcal{J}_1^{\infty}[g]$ and $\mathcal{J}_2^{\infty}[g]$ are also given as
\begin{equation*}
    \begin{split}
        \mathcal{J}_1^{\infty}[g](s) :=& - \left(- \Delta_{\Gamma_1} \right)^{1/2} g(s), \quad \mathcal{J}_2^{\infty}[g](s) := 0,
    \end{split}
\end{equation*}
where $\left(- \Delta_{\Gamma_1} \right)^{1/2} g(s)$ is the fractional Laplacian of order $1/2$ defined on $g$.

In the special case of $\sigma = \mu$, $A(x)$ is also isotropic in $\Omega_\delta$, which was studied by Li and Wang \cite{LW} in two dimensions. They derived EBCs that contain not only the usual Dirichlet, Neumann and Robin boundary conditions, but also some unusual ones including a Poisson equation and an integral equation. Because of the anisotropy of the thin layer in this paper, new EBCs which are not included in \cite{LW} emerge, involving the Dirichlet-to-Neumann mapping and the fractional Laplacian. Furthermore, we have to develop new estimates, and our approach is based on an auxiliary function originating from a harmonic extension. 

\begin{remark}
The condition $\sigma\delta^3\to 0$ in Table \ref{tb1} needs the smooth assumption of  $\Gamma_1 \in C^3$ and is necessary in the case that $\mu/\sigma \to 0$ as $\delta \to 0$. However, it is worth mentioning that we can eliminate this condition if $\mu / \sigma \to c \in (0,\infty]$ as $\delta \to 0$.
\end{remark}

\medskip

The past few decades witnessed developments in the idea of using EBCs, which produced lots of interesting results.  It was first recorded in the classic book of Carslaw and Jaeger \cite{HJ} in 1959 when they considered the heat equation in some simple cases. Subsequently, Sanchez-Palencia \cite{SP} first used such an idea rigorously in 1974 to study the interior reinforcement problem for elliptic and parabolic equations with a thin diamond-shaped inclusion layer. Later on,  Brezis, Caffarelli, and Friedman \cite{BCF} studied the elliptic problem in both interior and boundary reinforcement cases in 1980, followed by Buttazzo and Kohn \cite{BK} for the case of the rapid oscillating thickness of the coating. For more on the Poisson and heat equation, see \cite{CPW,LZ,LWZZ,LW2017,LSWW2021,LWZZ2009}. Furthermore, there is also a review paper of Wang \cite{XW} that paints a much more complete picture of the subject.

\medskip

The organization of this paper is as follows. Section \ref{sec 2} is devoted to establishing some a priori estimates and presenting some regularity results for the weak solution of (\ref{PDE}). Section \ref{sec 3} develops an auxiliary function via a harmonic extension, a bounded solution to an elliptic problem. In Section \ref{sec 4}, we apply the auxiliary function in Section \ref{sec 3} to problem (\ref{PDE}), deriving effective boundary conditions on $\Gamma_1 \times (0, T)$.

\section{Asymptotic behavior of the weak solution of (\ref{PDE})}\label{sec 2}
In this section, we investigate the asymptotic behavior of the weak solution of (\ref{PDE}) as the thickness of the layer shrinks.

\subsection{Weak solutions}
Before proceeding further,  we first define the weak solution of (\ref{EPDE}) with the effective boundary conditions mentioned in Table \ref{tb1} since there are some novel and unconventional results.  Although some of them have been discussed  in \cite{LW}, for the convenience of readers,  here we deal with these boundary conditions as well as the others.

\begin{definition}\label{def2}
$(1)$ $v$ is said to be a weak solution of (\ref{EPDE}) satisfying the boundary condition $v_1 = v_2$ together with $ k_1\frac{\partial v_1}{\partial \textbf{n}}-k_2\frac{\partial v_2}{\partial \textbf{n}}= \beta \Delta_{\Gamma_1} v $, if $v\in V^{1,0}_{2,0}(Q_T)$, and its trace on $\Gamma_1\times(0, T)$ belongs to $L^2 \left((0, T); H^1(\Gamma_1)\right)$, and it holds that
\begin{equation*}
\begin{aligned}
  && \mathcal{L}[v, \xi]& := -\int_{\Omega}u_0\xi(x,0)dx+\int_{Q_T} \left( A_0(x) \nabla v \cdot \nabla\xi  -v \xi_t-f\xi \right)dxdt\\
  &&\ &=-\beta \int_0^T\int_{\Gamma_1}\nabla_{\Gamma_1} v \cdot \nabla_{\Gamma_1} \xi dsdt,
\end{aligned}
\end{equation*}
for all $\xi\in C^\infty(\overline{Q}_T) $ with $\xi=0$ at $t=T$ and near $S_T.$

$(2)$ $v$ is said to be a weak solution of (\ref{EPDE}) satisfying the boundary condition $ v_1 =  v_2 $ together with $\nabla_{\Gamma_1} v =0$ and $\int_{\Gamma_1} \left(k_1\frac{\partial v_1}{\partial \textbf{n}}-k_2\frac{\partial v_2}{\partial \textbf{n}} \right) ds=0 $, if $v \in V^{1,0}_{2,0}(Q_T)$ and if for almost every fixed $t \in (0, T)$, its trace on $\Gamma_1$ is a constant, and it holds that
$$\mathcal{L}[v, \xi] = 0,$$
for all $\xi\in C^\infty(\overline{Q}_T)$ with $\xi=0$ at $t=T$ and near $S_T$, and $\nabla_{\Gamma_1} \xi = 0$ on $\Gamma_1$.

$(3)$ $v$ is said to be a weak solution of (\ref{EPDE}) satisfying the boundary condition $\nabla_{\Gamma_1} v_1 = \nabla_{\Gamma_1} v_2 =0$ together with $\int_{\Gamma_1} \left(k_1\frac{\partial v_1}{\partial \textbf{n}}-k_2\frac{\partial v_2}{\partial \textbf{n}}\right) ds = 0$ and $\int_{\Gamma_1} \left(k_1\frac{\partial v_1}{\partial \textbf{n}}-b(v_2 - v_1) \right)ds = 0$ for $b \in [0, \infty )$, if $v \in \widetilde{V}^{1,0}_{2,0}(Q_T)$ and if for almost every $t \in (0, T)$, its trace on $\Gamma_1$ is a constant, and it holds that
$$\mathcal{L}[v, \xi]= -  b\int_0^T\int_{\Gamma_1} (v_2-v_1)(\xi_2 - \xi_1) dsdt,$$
for any $\xi\in \widetilde{W}^{1,1}_{2,0}(Q_T) $with $\xi=0$ at $t=T$, and $\nabla_{\Gamma_1} \xi_1 = \nabla_{\Gamma_1} \xi_2 = 0$ on $\Gamma_1$.

\smallskip

$(4)$ $v$ is said to be a weak solution (\ref{EPDE}) satisfying the boundary condition 
$ k_1\frac{\partial v_1}{\partial \textbf{n}}=k_2\frac{\partial v_2}{\partial \textbf{n}}$ together with $b(v_2-v_1)=k_1\frac{\partial v_1}{\partial\textbf{n}}$, if $v\in  \widetilde{V}^{1,0}_{2,0}(Q_T)$, and it holds that
\begin{equation*}
    \mathcal{L}[v,\xi] = - b \int_0^T\int_{\Gamma_1} (\xi_2 - \xi_1)(v_2-v_1) dsdt,
\end{equation*}
for any test function $\xi\in \widetilde{W}^{1,1}_{2,0}(Q_T) $ with $\xi=0$ at $t=T$.

\smallskip

$(5)$ $v$ is said to be a weak solution of (\ref{EPDE}) satisfying the boundary condition  $k_1\frac{\partial v_1}{\partial \textbf{n}} = \gamma \mathcal{J}_1^{H} [v_1]-\gamma \mathcal{J}_2^{H} [v_2]$ together with
      $ k_2\frac{\partial v_2}{\partial \textbf{n}} =\gamma \mathcal{J}_2^{H} [v_1]-\gamma \mathcal{J}_1^{H} [v_2]$, if $v\in  \widetilde{V}^{1,0}_{2,0}(Q_T)$, where $H = \beta/\gamma$ or $\infty$, and it holds that
$$\mathcal{L}[v,\xi]= \gamma \int_0^T\int_{\Gamma_1}  \left( \mathcal{J}_1^{H} [v_1]-\mathcal{J}_2^{H} [v_2] \right)\xi_1 - \left( \mathcal{J}_2^{H} [v_1] - \mathcal{J}_1^{H} [v_2]\right)\xi_2 dsdt,
$$
for any test function $\xi\in \widetilde{W}^{1,1}_{2,0}(Q_T) $ with $\xi=0$ at $t=T$.
\end{definition}

The following theorem deals with the existence and uniqueness of the weak solution of (\ref{EPDE}) with the boundary conditions in Table \ref{tb1}.

\begin{theorem}
Suppose that $\Gamma_1 \in C^1$, $f \in L^2(Q_T)$, and $u_0 \in  L^2(\Omega) $. Then, $(\ref{EPDE})$ with any boundary condition in Table \ref{tb1} has one and only one weak solution $v$ as defined in Definition \ref{def2}.
\end{theorem}
\begin{proof}
The theorem can be proved by using the abstract parabolic theory. For a rigorous proof of this theorem, the reader can refer to \cite{LW} and \cite{CPW} (see also \cite{LM1972} and \cite{W1987}), and hence we omit the details.
\end{proof}

\medskip

Now, define a map $ X$ by
\begin{equation*}
    \Gamma_1 \times (0, \delta) \mapsto x = X(s, r) = \textbf{p}(s) + r \textbf{n}(s) \in \mathbb{R}^3,
\end{equation*}
where $\textbf{p}(s)$ is the projection of $x$ on $\Gamma_1$; $\textbf{n}(s)$ is the unit normal vector of $\Gamma_1$ pointing out of $\Omega_1$ at $\textbf{p}(s)$; $r$ is the distance from $x$ to $\Gamma_1$.

It is well known (\cite{GT}, Lemma 14.16) that for a small $\delta>0$, $X$ is a $C^1$ smooth diffeomorphism from $\Gamma_1 \times (0, \delta)$ to $\Omega_\delta$; $r = r(x)$ is a $C^2$ smooth function of $x$, which is the inverse of the mapping $x = X(s, r)$. Since $\Gamma_1$ is $C^2$ smooth, we parameterize the surface $\Gamma_1$ by a finite number of local charts with standard compatibility conditions. By using local coordinates $s=(s_1,s_2)$ in a typical chart on $\Gamma_1$, it holds  in $\overline{\Omega}_\delta$  that
\begin{equation*}\label{curvilinear}
    x=X(s,r)=X(s_1,s_2,r), \quad dx=[1+2H(s)r+\kappa(s)r^2]dsdr,  \tag{2.1}
\end{equation*}
where $ds$ represents the surface element; $H(s) $and $\kappa(s) $ represent mean curvature and Gaussian curvature at  $p$ on $\Gamma_1$, respectively. In the curvilinear coordinates, the Riemannian metric tensor at $x \in \overline{\Omega}_\delta$ induced from $ \mathbb{R}^3$ is defined as $G( s,r)$ with elements 
 $$g_{ij}(s, r)=g_{ji}(s, r) = < X_{i}, X_{j} >_{\mathbb{R}^3}, \quad i,j = 1,2,3,
 $$
where $X_i = X_{s_i}$ for $i =1,2$ and $X_3 = X_{r}$. Denote $| G |:= det G$ and let $g^{ij}(s,r)$ be the element of the inverse matrix of $G$, denoted by $G^{-1}$. 

Therefore, the derivatives in curvilinear coordinates and the formula of $A(x)$ in $\overline{\Omega}_\delta$ are presented as follows:
\begin{equation*}\label{derivative}
  \begin{split}
    \nabla u&=u_r \textbf{n} +\nabla_s u , \\
  \nabla_s u  :=\sum_{i,j=1,2} g^{ij}(s,r) u_{s_j}X_{s_i}(s,r)& \quad \text{ and } \quad \nabla_{\Gamma_1}u :=\sum_{i,j=1,2} g^{ij}(s,0)u_{s_j}\textbf{p}_{s_i}(s) ; 
\end{split} \tag{2.2}
\end{equation*}

\begin{equation*}\label{derivative2}
  \begin{split}
     \nabla \cdot \left(A(x)\nabla u\right) & = \frac{\sigma}{\sqrt{|G|}}\left(\sqrt{|G|}u_r \right)_r+\mu\Delta_{s}u,\\ 
 \Delta_{s}u = \nabla_s \cdot \nabla_s u =\frac{1}{\sqrt{|G|}}\sum_{ij=1,2}&\left(\sqrt{|G|}g^{ij}(s,r)u_{s_i}\right)_{s_j} \text{ and }  \Delta_{\Gamma_1}u = \nabla_{\Gamma_1} \cdot \nabla_{\Gamma_1} u\\
 A(x) = \sigma \textbf{n}(p)\otimes\textbf{n}(p)& + \mu \sum_{ij}g^{ij}(s,r)X_{s_i}(s,r)\otimes X_{s_j}(s,r).
\end{split} \tag{2.3}
\end{equation*}

\subsection{A priori estimates}
With the aid of the preceding curvilinear coordinates, we establish the following two estimates in this subsection. 

\smallskip

For ease of notation, let $C(T)$ denote a generic positive constant that solely depends on $T$, and let $O(1)$ also represent a quantity that may change from line to line but is independent of $\delta$.

\begin{lemma}\label{est1}
Suppose $f \in L^2(Q_T)$ and $u_0 \in L^2(\Omega).$ Then, any weak solution $u$ of (\ref{PDE}) satisfies the following inequalities.
\begin{equation*}
 \begin{split}
    (i)&\max_{t\in[0,T]}\int_{\Omega}u^2(x,t)dx+\int_{Q_T}\nabla u \cdot A\nabla u dxdt \leq C(T) \left(\int_{\Omega}u_0^2dx +\int_{Q_T}f^2dxdt \right),\\
    (ii) &\max_{t\in[0,T]} t\int_{\Omega}\nabla u \cdot A \nabla u dx+\int_{Q_T}t u_t^2dxdt\leq C(T)\left(\int_{\Omega}u_0^2dx +\int_{Q_T}f^2dxdt\right).
 \end{split}
\end{equation*}
\end{lemma}

\begin{proof}
Both estimates can be proved formally by a standard technique. Multiplying (\ref{PDE}) by $u$ and $t u_t$ separately, we use integration by parts in both the $t$ and the $x$ variables. Then, the same analysis on the Galerkin approximation of $u$, the details of which we omit, suggests that this argument can be made rigorous.
\end{proof}


\medskip

Next, fix a small $d$ such that $0 < 2 \delta< d$, and define 
\[
\widetilde{\sigma}=\begin{cases}
                  k_1,&   \quad -d\leq r\leq 0,  \\
               \sigma, &   \quad 0<r<\delta, \\  
              k_2,&  \quad \delta \leq r <d,
  \end{cases} \quad
\widetilde{\mu}=\begin{cases}
                  k_1,& \quad -d\leq r\leq 0, \\
               \mu, &   \quad 0<r<\delta, \\  
              k_2,&  \quad \delta \leq r <d.
  \end{cases}
\]
Denote the domain $\Omega_0 $ = $\Gamma_1\times(-d, d)$.
We are proceeding to provide higher order estimates.
\begin{lemma}\label{est2}
Suppose that $\Gamma_1 \in C^3$, $f \in L^2(Q_T)$, and $u_0 \in L^2(\Omega). $      Then, for any fixed $t_0>0$, the weak solution $u$ of (\ref{PDE}) satisfies 
\begin{equation*}\label{lemm2.21}
    \int^T_{t_0}\int_{\Gamma_1} \int_{-d}^d \left(\widetilde{\mu} (\Delta_s u)^2 + \widetilde{\sigma } \left|\nabla_s (u_r) \right|^2 \right) drdsdt \leq O(1) \left(1 + \frac{ \widetilde{\sigma} }{ \widetilde{\mu} } + \frac{1}{\widetilde{\mu}} \right), \tag{2.4}
\end{equation*}
and
\begin{equation*}\label{lemm2.22}
    \int^T_{t_0}\int_{\Gamma_1} \int_{-d}^{d}\widetilde{\sigma} u^2_{rr} drdsdt \leq O(1) \left(1 + \frac{1}{\widetilde{\sigma}} + \frac{ \widetilde{\mu} }{\widetilde{\sigma} } \right).\tag{2.5}
\end{equation*}
\end{lemma}
     
\begin{proof}
Based on the above assumption of $d$, we take a smooth cut-off function $\eta$ satisfying
\[
\eta=\eta(r)=\begin{cases}
                  1, &  \text{if} \quad |r|\leq \frac{d}{2},  \\
                  0, &  \text{if} \quad |r| \geq \frac{3d}{4},
  \end{cases}
\]
with $0 \leq \eta \leq 1$ and $|\eta^{\prime}| \leq \frac{C}{d}$ for a positive constant $C$ independent of $\delta$. According to the curvilinear coordinates $(s,r)$  and the derivatives in (\ref{derivative2}) near $\Gamma_1$, we are led to
\begin{equation*}\label{eq24}
    u_t- \left(\frac{\widetilde{\sigma}}{\sqrt{|G|}}(\sqrt{|G|}u_r)_r+\widetilde{\mu}\Delta_{s}u\right)=f, \tag{2.6}
\end{equation*}
for $s \in \Gamma_1$ and $r \in (-d, d)\backslash \{0, \delta\}$.

For any $t_0>0$, using (\ref{curvilinear}),  we multiply both sides of (\ref{eq24}) by $\eta^2 \Delta_{s} u$ and integrate (\ref{eq24}) in both the $x$ and $t$ variables, leading to 
\begin{equation*}\label{eq25}
\begin{split}
     &\int_{t_0}^T\int_{\Gamma_1}\int_{-d}^{d}
     (u_t-f)\eta^2\Delta_{s} u F(s,r) drdsdt\\
  =&\int_{t_0}^T\int_{\Gamma_1}\int_{-d}^{d}\left(\frac{\widetilde{\sigma}}{\sqrt{|G|}}\left(\sqrt{|G|}u_r\right)_r+\widetilde{\mu}\Delta_{s}u\right)\eta^2 \Delta_{s} u F(s,r) drdsdt,
 \end{split} \tag{2.7}
\end{equation*}
where $F(s,r):=1+2H(s)r+\kappa(s)r^2$. Since $H(s)$ and $\kappa(s)$ are uniformly bounded as $\delta$ is sufficiently small, it follows from Young's inequality that
\begin{equation*}\label{eq26}
\begin{split}
     &\int_{t_0}^T\int_{\Gamma_1}\int_{-d}^{d}
     (u_t-f)\eta^2\Delta_{s} u F(s,r) drdsdt\\
  \leq & \frac{1}{2}\int_{t_0}^T\int_{\Gamma_1}\int_{-d}^{d} \widetilde{\mu} \eta^2 (\Delta_{s}u)^2 drdsdt + O(1) \int_{t_0}^T\int_{\Gamma_1}\int_{-d}^{d} 
  \frac{(u_t^2+f^2)}{\widetilde{\mu}}\eta^2 drdsdt.
 \end{split} \tag{2.8}
\end{equation*}

Reviewing the first term on the right-hand side of (\ref{eq25}), we get
\begin{equation*}\label{eqno27}
\begin{split}
    & \int_{t_0}^T\int_{\Gamma_1}\int_{-d}^{d}\frac{\widetilde{\sigma}}{\sqrt{|G|}}\left(\sqrt{|G|}u_r\right)_r\eta^2 \Delta_{s} u F(s,r) drdsdt\\
    =&  \int_{t_0}^T\int_{\Gamma_1}\int_{-d}^{d}\widetilde{\sigma}\left(u_{rr}-\left(\frac{1}{\sqrt{|G|}} \right )_r\sqrt{|G|}u_r\right)\eta^2 \Delta_{s} u F(s,r) drdsdt. 
\end{split} \tag{2.9}
\end{equation*}

Performing integration by parts in the $s$ variable and Young's inequality, we obtain
\begin{equation*}\label{eqno210}
\begin{split}
   &-\int_{t_0}^T\int_{\Gamma_1}\int_{-d}^{d}\widetilde{\sigma} \left(\frac{1}{\sqrt{|G|}} \right)_r\sqrt{|G|}u_r\eta^2\Delta_{ s }u F(s,r) drdsdt\\
  \geq &-\frac{1}{8}\int_{t_0}^T\int_{\Gamma_1}\int_{-d}^{d}\widetilde{\sigma}\eta^2 \left| \nabla_{s}(u_r)\right|^2  drdsdt - O(1) \int_{t_0}^T\int_{\Gamma_1}\int_{-d}^{d}\widetilde{\sigma}\eta^2 \left(u_r^2+ |\nabla_s u|^2 \right)drdsdt.
\end{split}\tag{2.10}
\end{equation*}

In addition to (\ref{eqno210}), since $\left(\nabla_s (u_r)\right)_{r} = \nabla_s(u_{rr}) + \underset{ij=1,2}{\sum}\left(g^{ij}(s,r)X_{s_i}\right)_r u_{rs_i}$, the integration by parts and Young's inequality also give rise to
\begin{equation*}\label{eqno211}
\begin{split}
   &\int_{t_0}^T\int_{\Gamma_1}\int_{-d}^{d}\widetilde{\sigma} u_{rr} \eta^2 \Delta_s u F(s,r) drdsdt\\
  =&-\int_{t_0}^T\int_{\Gamma_1}\int_{-d}^{d}\widetilde{\sigma} \eta^2 \nabla_s u \cdot \nabla_s \left( F(s,r) \right) u_{rr}drdsdt\\
  &+\int_{t_0}^T\int_{\Gamma_1}\int_{-d}^{d}\widetilde{\sigma} \left[\left(\eta^2 F(s,r)\right)_r \nabla_s u + \eta^2 F(s,r) (\nabla_s u)_r\right] \cdot \nabla_s (u_{r}) drdsdt\\
  & +\int_{t_0}^T\int_{\Gamma_1}\int_{-d}^{d}\widetilde{\sigma} \eta^2 \nabla_s u \cdot \left(\underset{ij=1,2}{\sum}\left(g^{ij}(s,r)X_{s_i}\right)_r u_{rs_i} \right)F(s,r) drdsdt\\
  := & I + II +III,
\end{split}\tag{2.11}
\end{equation*}
where we need to modify the assumption $\Gamma_1 \in C^2$ to $\Gamma_1 \in C^3$ because (\ref{eqno211}) involves the first derivative of $\kappa(s)$ and $H(s)$.

\smallskip

Subsequently, we consider the right-hand side of (\ref{eqno211}) separately due to its tediousness. By using integration by parts, some tedious manipulation yields
\begin{equation*}\label{R1}
\begin{split}
   I=&-\int_{t_0}^T\int_{\Gamma_1}\int_{-d}^{d}\widetilde{\sigma} \eta^2 \nabla_s u \cdot \nabla_s \left( F(s,r) \right) u_{rr}drdsdt\\
  = & -\int_{t_0}^T\int_{\Gamma_1}\int_{-d}^{d}\widetilde{\sigma} u_{r} \left\{\left[\eta^2 \nabla_s\left( F(s,r) \right)\right]_r \cdot \nabla_s u  + \eta^2 \nabla_s\left( F(s,r) \right) \cdot  (\nabla_s u)_r \right\} drdsdt\\
  \geq &-\frac{1}{8}\int_{t_0}^T\int_{\Gamma_1}\int_{-d}^{d}\widetilde{\sigma} \eta^2 \left|\nabla_s (u_r)\right|^2  drdsdt - O(1) \int_{t_0}^T\int_{\Gamma_1}\int_{-d}^{d}\widetilde{\sigma}\eta^2 \left(u_r^2+ |\nabla_s u|^2 \right)drdsdt,
\end{split}\tag{2.12}
\end{equation*}
where we have used the fact that $\left(\nabla_s u\right)_r = \nabla_s(u_r) + \underset{ij=1,2}{\sum}\left(g^{ij}(s,r)X_{s_i}\right)_r u_{s_i}$ and the transmission conditions (\ref{trans}). 

\smallskip

Thanks to the transmission conditions (\ref{trans}) and integration by parts again, we arrive at
\begin{equation*}\label{R2}
\begin{split}
  II +III 
  \geq  &  \frac{1}{2}\int_{t_0}^T\int_{\Gamma_1}\int_{-d}^{d}\widetilde{\sigma} \eta^2 \left|\nabla_s (u_r)\right|^2  drdsdt - O(1) \int_{t_0}^T\int_{\Gamma_1}\int_{-d}^{d}\widetilde{\sigma}\eta^2  |\nabla_s u|^2 drdsdt,
 \end{split} \tag{2.13}
\end{equation*}
where we made use of the positive definiteness of $G$. 
\smallskip

In view of (\ref{eq26})-(\ref{R2}), it follows from Lemma \ref{est1} that
\begin{equation*}\label{eq214}
\begin{split}
    &\int_{t_0}^T\int_{\Gamma_1}\int_{-d}^{d}\eta^2 \left(\widetilde{\mu}(\Delta_{s} u)^2 +\widetilde{\sigma} \left|\nabla_{s}(u_r)\right|^2\right) drdsdt
    \leq O(1)\left(1+\frac{1}{\widetilde{\mu}}+\frac{\widetilde{\sigma}}{\widetilde{\mu}} \right),
\end{split}\tag{2.14}
\end{equation*}
that is, the assertion (\ref{lemm2.21}) holds. 

\smallskip

Thus, we are left to handle the remaining term $u_{rr}$. Since 
\begin{equation*}
    \widetilde{\sigma}u_{rr}=u_t-f-\frac{\widetilde{\sigma}}{\sqrt{|G|}}(\sqrt{|G|})_ru_r-\widetilde{\mu}\Delta_s u
\end{equation*}
near $\Gamma_1$, it turns out that
\begin{equation*}\label{eq215}
\begin{split}
   \widetilde{\sigma}u_{rr}^2&\leq  O(1)(u_t^2+f^2)+O(1)\left(\widetilde{\sigma} u_r^2+\frac{\widetilde{\mu}^2}{\widetilde{\sigma}}(\Delta_s u)^2\right).\\ 
\end{split}\tag{2.15}
\end{equation*}
Combining this with (\ref{eq214}), we find that
\begin{equation*}\label{eq211}
\begin{split}
   \int_{t_0}^T\int_{\Gamma_1}\int_{-d}^{d}\widetilde{\sigma} u_{rr}^2 drdsdt
  \leq & O(1)\int_{t_0}^T\int_\Omega (u_t^2+f^2)dxdt+O(1)\int_{t_0}^T\int_{\Gamma_1}\int_{-d}^{d}\widetilde{\sigma} u_r^2+\widetilde{\mu}(\Delta_s u)^2  drdsdt\\   
  \leq & O(1) \left(1+\frac{1}{\widetilde{\sigma}}+\frac{\widetilde{\mu}}{\widetilde{\sigma}} \right).
\end{split}\tag{2.16}
\end{equation*}

This completes the proof of Lemma \ref{est2}.
\end{proof}

\medskip

Enlightened by \cite{CPW}, we finally conclude this section by addressing the regularity issue of the weak solution of (\ref{PDE}) for even general coefficients $A = A( x,t) =  \left(a_{ij}(x,t) \right)_{N \times N}$. Let (\ref{PDE}) be a strictly parabolic problem, in which sense, $a_{ij}(x,t)$ satisfies 
\begin{equation*}
    \underset{i,j}{\sum} a_{ij}(x,t) \xi_i \xi_j \geq \lambda_0 |\xi|^2,
\end{equation*}
for any $\xi \in \mathbb{R}^N$ and some constant $\lambda_0>0$. 

\smallskip

Instead of trying to be rigorous with details, the intention is to highlight the regularity results in the following theorem.

\begin{theorem}\label{thm2}
Let $k$ be an integer with $k \geq 2$ and $\alpha \in (0,1).$ Suppose that $\Gamma_1\in C^{k +\alpha}$  and  $f \in C^{k -2+\alpha,(k-2+\alpha)/2}(\overline{\Omega}_h\times[0,T])$, where $h = 1, \delta, 2$. 

\smallskip

If $a_{ij}\in C^{k-1+\alpha,(k-1+\alpha)/2}(\overline{\Omega}_h\times[0, T])$, then for any $t_0>0$, the weak solution $u$ of (\ref{PDE}) satisfies
\begin{equation*}
    u\in C^{k +\alpha, (k+\alpha)/2} \left(\overline{\mathcal{N}}_i\times [t_0, T] \right),
\end{equation*}
where $\mathcal{N}$ is a narrow neighborhood of $\partial\Omega_\delta (\Gamma_1 \cup \Gamma_2)$ and $\mathcal{N}_i =\mathcal{N}\cap \Omega_i$.
\end{theorem}
\begin{proof}
In the interior of $\Omega_1$, $\Omega_\delta$, $\Omega_2$, and near $\Gamma_1 \cup \Gamma_2$, we have standard $L^p$ and Schauder regularity theories, whereas the regularity across the interface $\Gamma_1 \cup \Gamma_2$ is not straightforward.
 However, the theorem can be proved by the same method as employed in \cite{CPW}, which uses the idea of Nirenberg as recorded in \cite{LV} together with $L^p$ and Schauder theories for parabolic systems. So, we omit the details.
\end{proof}

\section{Auxiliary functions}\label{sec 3}
This section intends to construct an auxiliary function and estimate its behavior as the thickness of the thin layer is sufficiently small. Our idea of developing the auxiliary function is adapted from \cite{CPW} via a harmonic extension.  

\subsection{Estimates for the auxiliary function}
Let $\xi$ be a test function for (\ref{PDE}) satisfying $\xi \in C^\infty \left(\overline{\Omega} \times [ 0, T] \right)$ with $\xi(x, T) = 0$ and vanishing near $S_T$. Our next goal is to reconstruct a new test function $\overline{\xi}$ which differs from $\xi$ in $\Omega_\delta$. This procedure requires the support of the auxiliary function defined in the sequel.

\smallskip

For every $t\in [0, T]$, let $ \psi(s,r,t)$ be a bounded solution of 
\begin{equation*}\label{AF}
    \left\{
             \begin{array}{ll}
                 \sigma\psi_{rr}+ \mu \Delta_{\Gamma_1} \psi=0, & \Gamma_1\times(0, \delta),\\
                  \psi (s, 0, t)=g_1(s), & 
                  \psi(s, \delta, t)=g_2(s),
             \end{array}
  \right. \tag{3.1}
\end{equation*}
where $g_1(s) :=\xi(s, 0, t)$ and $g_2(s) :=\xi(s, \delta, t)$.

\smallskip

A little manipulation is needed to eliminate the coefficients $\sigma$ and $\mu$. Let $r = R\sqrt{\sigma/\mu}$, and suppressing the time dependence and substituting $r$ into (\ref{AF}), we have 
\begin{equation*}
    \Psi(s, R) = \psi(s, R\sqrt{\sigma/\mu}, t),
\end{equation*}
resulting in
\begin{equation*}\label{resca}
    \left\{
             \begin{array}{ll}
                   \Psi_{RR}+\Delta_{\Gamma_1} \Psi=0,  & \Gamma_1\times(0, h),\\
                  \Psi(s, 0)=g_1(s), &      \Psi(s, h)=g_2(s),
             \end{array}
  \right.\tag{3.2}
\end{equation*}
with
\begin{equation*}
    h:=\delta\sqrt{\frac{\mu}{\sigma}}=\frac{\mu\delta}{\sqrt{\sigma\mu}}=\frac{\sqrt{\sigma\mu}}{\sigma/\delta} . \tag{3.3}
\end{equation*}

As is known, for fixed $\delta>0$, the existence and uniqueness of the bounded solution $\Psi(s, R)$ follow from the standard elliptic theory. 

If $h \to H \in (0, \infty)$ as $\delta \to 0$, then we define the operator $ \mathcal{J}^H$  by
\begin{equation*}
    \mathcal{J}^H[g_1](s) :=\lim_{\delta\to 0}\Psi_R(s, 0),
\end{equation*}
where $ \mathcal{J}^H[g_1](s)$ is known as the Dirichlet-to-Neumann mapping that transforms a Dirichlet condition into a Neumann condition. The rigorous formula of that is also given but postponed in the coming subsection.

Since the rescaling relationship between $\psi$ and $\Psi$ becomes
\begin{equation*}\label{eqno34}
    \sigma     \psi_r(s,r,t)=\sqrt{\sigma\mu}      \Psi_R(s,R), \tag{3.4}
\end{equation*}
multiplying both sides of (\ref{AF}) by $u$ and using integration by part, afterward, we acquire
\begin{equation*}\label{eq35}
    \begin{split}
      -\int_{\Gamma_1}\int_0^\delta \left(\sigma  \psi_r u_r+\mu \nabla_{\Gamma_1}  \psi \cdot \nabla_{\Gamma_1}  u \right)ds dr
      =&-\sigma\int_{\Gamma_1} \left( \psi_r(s, \delta, t) u(s, \delta, t) - \psi_r(s, 0, t) u(s, 0, t) \right)ds\\
      =&-\sqrt{\sigma\mu}\int_{\Gamma_1} \left(\Psi_R(s,h)u(s,\delta,t) -\Psi_R(s,0)u(s,0,t)\right) ds.
      \end{split}\tag{3.5}
\end{equation*}
Applying the same performance above on (\ref{AF}) by changing $u$ to $\psi$, we get
\begin{equation*}\label{eqno36}
    \begin{split}
      -\int_{\Gamma_1}\int_0^\delta \left(\sigma  \psi_r^2 +\mu |\nabla_{\Gamma_1}  \psi |^2  \right)ds dr
      =&-\sigma\int_{\Gamma_1} \left( \psi_r(s, \delta, t) g_2(s) - \psi_r(s, 0, t) g_1(s) \right)ds\\
      =&-\sqrt{\sigma\mu}\int_{\Gamma_1} \left(\Psi_R(s, h)\xi(s, \delta, t) - \Psi_R(s, 0) \xi(s, 0, t) \right) ds.
      \end{split}\tag{3.6}
\end{equation*}

\medskip

We are now in a position to estimate $\Psi_R(s, 0)$ and $\Psi_R(s, h)$. However, because $h$ relies on the value of $\delta$, we need to handle the problem separately. 

If $h$ is sufficiently small with $h\to 0$ as $\delta\to 0$, then it follows from the formula of $\Psi_R(s, R)$ in the next subsection that
\begin{equation*}\label{eq37}
\begin{split}
    \left| \Psi_R(s,0)-\frac{g_2(s)-g_1(s)}{h} \right| & \leq O(h).
\end{split}\tag{3.7}
\end{equation*}
Hence, we have
\begin{equation*}\label{size1}
  \begin{split}
      \Psi_R(s,0)&= \frac{g_2(s)-g_1(s)}{h} + O(h),
  \end{split}\tag{3.8}
\end{equation*}
and also, 
\begin{equation*}\label{size11}
  \begin{split}
      \Psi_R(s,h)= \frac{g_2(s)-g_1(s)}{h} + O(h).
  \end{split}\tag{3.9}
\end{equation*}

Likewise, by the formulas of $\Psi(s, 0)$ and $\Psi(s, R)$ given in the future, we obtain
\begin{equation*}\label{eqno310}
   \begin{split}
      \left|\Psi_R(s,h) - \Psi_R(s,0) + \frac{h}{2} \Delta_{\Gamma_1} g_1(s) + \frac{h}{2} \Delta_{\Gamma_1} g_2(s) \right|  \leq O(h^2).
   \end{split}\tag{3.10}
\end{equation*}
Subsequently, (\ref{eqno310}) gives
\begin{equation*}\label{size2}
     \Psi_R(s, h)-\Psi_R(s,0)= -\frac{h}{2} \Delta_{\Gamma_1} g_1(s) - \frac{h}{2} \Delta_{\Gamma_1} g_2(s)  + O(h^2).\tag{3.11}
\end{equation*}

\smallskip

On the other hand, if $h \to H \in (0, \infty]$ as $\delta\to 0$, from the Taylor's expansion for $\Psi(s,R)$, then we obtain
\begin{equation*}
      \Psi_R(s,0)=\frac{  \Psi(s,R)-  \Psi(s,0)}{R}-\frac{R}{2}  \Psi_{RR}(s,\overline{R}),
\end{equation*}
for some $\overline{R}\in [0, R]$. Taking $R = min \{h,1\}$, we get
\begin{equation*}
\begin{split}
    \|  \Psi_R(s,0)\|_{L^\infty(\Gamma_1)}&\leq \frac{O(1)}{R},
\quad
   \|  \Psi_R(s,h)\|_{L^\infty(\Gamma_1)}\leq \frac{O(1)}{R}.
\end{split}
\end{equation*}
Consequently, (\ref{eqno34}) yields
\begin{equation*}\label{size3}
    \|\sigma  \psi_r(s,0,t)\|_{L^\infty(\Gamma_1)}=\sqrt{\sigma\mu} \|  \Psi_R(s,0)\|_{L^\infty(\Gamma_1)} = O(1)\sqrt{\sigma\mu}. \tag{3.12}
\end{equation*}

\subsection{The formula for the auxiliary function}
The task of this subsection is to find rigorous formulas for $\Psi_R(s, 0)$ and  $\Psi_R(s, R)$. The only method used is the separation of variables, by which means we derive
\begin{equation*}
    \Psi(s, R) = \sum_{n=0}^\infty \left(A_n e^{\sqrt{\lambda_n} R} + B_n e^{-\sqrt{\lambda_n} R} \right) e_n(s), 
\end{equation*}
where $A_n$ and $B_n$ are two coefficients independent of $R$; $\lambda_n$ and $e_n(s)$ are  the  eigenvalues and the corresponding eigenfunctions of the Laplacian-Beltrami operator $-\Delta_{\Gamma_1}$ defined on $\Gamma_1$.

\smallskip

According to the boundary conditions in (\ref{resca}), implementing simple calculation gives rise to 
\begin{equation*}\label{eq313}
    A_n = \frac{g_{2n}-g_{1n}e^{-\sqrt{\lambda_n} h}}{2 sinh(\sqrt{\lambda_n} h)}, \quad B_n = \frac{-g_{2n} +g_{1n}e^{\sqrt{\lambda_n} h}}{2 sinh(\sqrt{\lambda_n} h)}, \tag{3.13}
\end{equation*}
where
\begin{equation*}
    g_{1n} = < e_n, g_1 > := \int_{\Gamma_1} e_n(s) g_1(s) ds, \quad g_{2n} = <e_n, g_2> := \int_{\Gamma_1} e_n(s) g_2(s) ds.
\end{equation*}
Thus, we infer that 
\begin{equation*}\label{eq314}
\begin{split}
      &\Psi_R(s,0)=\sum_{n=1}^\infty \sqrt{\lambda_n} e_n(s) \frac{2g_{2n}-g_{1n}(e^{-\sqrt{\lambda_n} h}+e^{\sqrt{\lambda_n} h})}{2 sinh(\sqrt{\lambda_n} h)},\\
    & \Psi_R(s,h)= \sum_{n=1}^\infty \sqrt{\lambda_n} e_n(s) \frac{-2g_{1n}+g_{2n}(e^{-\sqrt{\lambda_n} h}+e^{\sqrt{\lambda_n} h})}{2 sinh(\sqrt{\lambda_n} h)}.
\end{split}\tag{3.14}
\end{equation*}

Before approaching deeper into the formula of $\mathcal{J}^H[g_1](s)$, we first focus on $g_2(s)$ since it depends on how we choose the test function $\xi$. 

Consider $\xi$ in the case that is independent of $\delta$ and $\underset{\delta \to 0}{\lim} g_2(s) = g_1(s)$. If $h\to H\in (0, \infty)$ as $\delta \to 0$, then it follows from (\ref{eq314}) that
\begin{equation*}\label{d2n}
\begin{split}
     \mathcal{J}^H[g_1](s) = \underset{\delta \to 0}{\lim}\Psi_R(s, 0)&=\sum_{n=1}^\infty \sqrt{\lambda_n} e_n(s) \frac{2g_{1n}-g_{1n}(e^{-\sqrt{\lambda_n} H}+e^{\sqrt{\lambda_n} H})}{2 sinh(\sqrt{\lambda_n} H)}.
\end{split}\tag{3.15}
\end{equation*}

Contrasted with the case mentioned above, if the test function $\xi$ remains to vanish in $\overline{\Omega}_2$, say, $g_2(s)\equiv 0$, we then define 
\begin{equation*}\label{d2n2}
\begin{split}
   \mathcal{J}_1^H[g_1](s) = \underset{\delta \to 0}{\lim} \Psi_R(s,0) &= \sum_{n=1}^\infty \sqrt{\lambda_n} e_n(s) \frac{-g_{1n}(e^{-\sqrt{\lambda_n} H}+e^{\sqrt{\lambda_n} H})}{2 sinh(\sqrt{\lambda_n} H)},\\
   \mathcal{J}_2^H[g_1](s) = \underset{\delta \to 0}{\lim} \Psi_R(s,H)&=\sum_{n=1}^\infty \sqrt{\lambda_n} e_n(s) \frac{-2g_{1n}}{2 sinh(\sqrt{\lambda_n} H)}.
\end{split}\tag{3.16}
\end{equation*}
Also, note that 
$$
\mathcal{J}^H[g_1](s) = \mathcal{J}_1^H[g_1](s) - \mathcal{J}_2^H[g_1](s).
$$
Furthermore, since 
\begin{equation*}
\begin{split}
     &\left|\mathcal{J}^{H_1}[g_1](s) - \mathcal{J}^{H_2}[g_1](s) \right|
     \leq C \left|H_1-H_2\right|\sum_{n=1}^\infty 2\lambda_n e_n(s) g_{1n}\frac{2 -e^{-\sqrt{\lambda_n} h^\prime}-e^{\sqrt{\lambda_n} h^\prime}}{ (e^{\sqrt{\lambda_n} h^\prime}-e^{-\sqrt{\lambda_n} h^\prime})^2},
\end{split}
\end{equation*}
for some $h^\prime \in (H_1, H_2)$, if $H_2 \to \infty$, then $h^\prime \to \infty$, which implies that
 $\mathcal{J}^H$ converges uniformly in $H$, and so do $\mathcal{J}_1^H$ and $\mathcal{J}_2^H$. 
 
 Moreover, we define 
 $$\mathcal{J}^\infty[g_1](s) = \underset{	H \to \infty}{\lim}\mathcal{J}^H[g_1](s), \quad \mathcal{J}_2^\infty[g_1](s) = \underset{ H \to \infty}{\lim}\mathcal{J}_2^H[g_1](s).
 $$
 Using (\ref{d2n2}) again, it is easy to check that $\mathcal{J}_2^\infty[g_1](s) = 0$ and
\begin{equation*}
     \begin{split}
       \mathcal{J}_1^\infty[g_1](s) = \mathcal{J}^\infty[g_1](s) := - (-\Delta_{\Gamma_1} )^{1/2}g_1(s),
     \end{split}
\end{equation*}
where $(-\Delta_{\Gamma_1} )^{1/2}g_1(s)$ is the fractional Laplacian on a smooth function $g_1$ defined on $\Gamma_1$.

\medskip

Finally, we end this section by pointing out that these Dirichlet-to-Neumann operators are linear and symmetric. If $g_1(s)$ is smooth, for $H \in (0, \infty]$, we then define a functional on a function $w \in H^{1/2}(\Gamma_1)$ by 
\begin{equation*}\label{eq317}
    \begin{split}
        F(w) = \int_{\Gamma_1} \mathcal{J}^H[g_1](s) w(s) ds =:  <\mathcal{J}^H[g_1], w>.
    \end{split}\tag{3.17}
\end{equation*}
Thus we extend this functional for general $g_1(s) \in H^{1/2}(\Gamma_1)$.

\smallskip

The proof of the linearity is trivial that we omit it. Because of (\ref{d2n2}), we can show the symmetry upon a direct computation 
\begin{equation*}\label{eq318}
    \begin{split}
        <\mathcal{J}^H[g_1], w> = \int_{\Gamma_1} \mathcal{J}^H[g_1](s) w(s) ds = \int_{\Gamma_1}  g_1(s)\mathcal{J}^H[w](s) ds = <g_1, \mathcal{J}^H[w]>.
    \end{split}\tag{3.18}
\end{equation*}

For $\mathcal{J}_1^H[g_1]$ and $\mathcal{J}_2^H[g_1]$,
the same procedure as above can be used to show that   they are also linear and symmetric.
\section{Proof of Theorem \ref{thm:bigthm} }\label{sec 4}
The main result of this section is to address effective boundary conditions on $\Gamma_1 \times (0, T)$.
\begin{proof}
The proof of Theorem \ref{thm:bigthm} consists of two significant steps. First, we establish a compactness argument on the weak solution of (\ref{PDE}) to show the strong convergence $u \to v$ after passing to a subsequence of $\delta \to 0$. Next, we find effective boundary conditions arising from Table \ref{tb1} with the trick of an auxiliary function $\psi(s, r, t)$ generated by a harmonic extension.

\medskip

\noindent\textbf{ Step 1}. To begin with the proof,  we first focus on the compactness of $\{ u \}_{\delta > 0}.$

For a sufficiently small $\delta$, we choose a small $d> \delta$ and consider the domain $\Omega_d =\{x \in \Omega_2 |  dist(x, \Gamma_1) >d \} $. From Lemma \ref{est1}, it is apparent that for a given small $t_0 \in (0,T]$, $\{u\}_{\delta >0}$ is bounded in $W^{1, 0}_2(\Omega_1 \times (0, T))$, $W^{1, 0}_2(\Omega_d \times (0, T))$, $W^{1, 1}_2(\Omega_1 \times (t_0, T))$ and $W^{1, 1}_2(\Omega_d \times (t_0, T))$. Thus, after passing to a subsequence of $\delta \to 0$,  $u \to v $ weakly in the above spaces; moreover, $v$ belongs to all these spaces.

Alternatively, we replace $d$ by $\{d_k\}_{k\geq 0}$ such that $d_0 = d$ and $d_k\to 0$ as $k \to \infty$. By a diagonal argument, after further passing to a subsequence of $\delta \to 0$, $u \to v$ weakly in $W^{1,1 }_2(\Omega_{d_k} \times (t_0, T))$ for all $k$. This implies that $v \in W^{1, 0}_2(Q_T^1)$, $W^{1,0}_2(Q^2 _T)$,  $W^{1, 1}_2(\Omega_1 \times (t_0, T))$ and $W^{1, 1}_2\left((\Omega\backslash\overline{\Omega}_1) \times (t_0, T) \right)$.

Furthermore, using Lemma \ref{est1}, we deduce that $\{u\}_{\delta>0}$ is also bounded in $C([t_0, T];  H^1(\Omega_1))$ and $C([t_0, T];  H^1(\Omega_{d_k}))$. According to the Banach-Eberlein theorem, $ u \to v$ weakly in $C([t_0, T]; H^1(\Omega_1))$ and $C([t_0, T]; H^1(\Omega_{d_k}))$ after passing to a subsequence of $\delta \to 0.$ Together with the compactness of the embedding $H^1(\Omega)\hookrightarrow L^2(\Omega)$, for any fixed $t$, $\{u\}_{\delta>0}$ is precompact in $L^2(\Omega_1)$ and $L^2(\Omega_{d_k})$ for all $k$.

Note that the functions $\{u\}_{\delta>0}$: $t\in[t_0, T]\mapsto u(\cdot, t)\in L^2(\Omega)$ are equicontinuous for the boundedness of the term $\int_{Q_T}t u_t^2dxdt$ in Lemma \ref{est1}. Consequently, the generalized Arzela-Ascoli theorem suggests that after passing to a further subsequence of $\delta \to 0$, $u \to v$ strongly in $C\left([t_0, T]; L^2(\Omega_1)\right)$ and $C\left([t_0, T]; L^2(\Omega_{d_k})\right)$. 

Therefore, by sending $k \to \infty$ afterward, we verify that  $u \to v$ strongly in $C\left([t_0,T];L^2(\Omega)\right)$, resulting from
\begin{equation*}
\begin{split}
    \int_{\Omega}(u-v)^2 dx &\leq \int_{\Omega_1}  (u-v)^2 dx + \int_{\Omega_{d_k}}  (u-v)^2 dx +\sqrt{d_k},
\end{split}
\end{equation*}
where Lemma \ref{est1} was used.

\smallskip

For the completion of the compactness argument, it now remains to prove that $u \to v$ strongly in $C\left([0, T]; L^2(\Omega) \right)$. To this end, we construct a sequence $u_0^n \in C_0^\infty(\Omega)$ in the following way: multiply $u_0$ by cut-off functions in the $r$ variable with $u_0^n$ vanishing in $\Omega_\delta$, satisfying $\|\nabla u_0^n\|_{L^2(\Omega)} \leq C(n)$ independent of $\delta$, and $\|u_0-u_0^n\|_{L^2(\Omega)}\leq \frac{1}{n}+\|u_0\|_{L^2(\Omega_\delta)}$. 

Then, we decompose $u=u_1+u_2$, where $u_1$ is the unique weak solution of (\ref{PDE}) with $f=0$ and the initial value changed by $u_0-u_0^n$, and $u_2$ is the unique weak solution of (\ref{PDE}) with the initial value changed by $u_0^n$. Performing energy estimates on the new PDE concerning $u_1$ results in
\begin{equation*}
\|u_1(\cdot, t)\|_{L^2(\Omega)} \leq \|u_0-u_0^n\|_{L^2(\Omega)} \leq \frac{1}{n}+\|u_0\|_{L^2(\Omega_\delta)}.
\end{equation*}

For any small $t\in [0, t_0]$, multiplying the PDE for $u_2$ by $(u_2)_t $ and performing integration by parts in both the $x$ and $t$ variables over $Q^{t}=\Omega\times(0, t)$, we then have
\begin{equation*}
\begin{split}
 \int_{Q^{t}}(u_2)^2_t dxdt+\int_\Omega \nabla u_2 \cdot A\nabla u_2 dx 
\leq & \int_{Q^{t}}f^2 dxdt+k_1\int_{\Omega_1} |\nabla u_0^n |^2dx+k_2 \int_{\Omega_2} |\nabla u_0^n |^2dx \leq C(n^2).
\end{split}
\end{equation*}

As a result, 
\begin{equation*}
\begin{aligned}
  &&  \|u_2(\cdot,t)- u_0^n (\cdot)\|^2_{L^2(\Omega)} &= 2\int_0^{t}\int_\Omega (u_2(x,t)-u_0^n(x)) (u_2)_tdxdt \\
  &&\ &\leq 2\left(\int_0^{t}\int_\Omega |u_2(x,t)-u_0^n(x)|^2 \right)^\frac{1}{2} \left(\int_0^{t}\int_\Omega (u_2)_t^2 \right)^\frac{1}{2}dxdt\\
  &&\ &\leq 2\sqrt{t} \max_{t\in[0,t]} \|u_2(\cdot,t)-u_0^n(\cdot)\| C(n).
\end{aligned}
\end{equation*}

Finally, combining the above estimates, we get
\begin{equation*}
\begin{split}
   \|u(\cdot, t)- u_0(\cdot)\|_{L^2(\Omega)} &\leq \|u_1(\cdot, t)\|_{L^2(\Omega)} +\|u_2(\cdot, t)- u_0(\cdot)\|_{L^2(\Omega)}+\|u_0- u_0(\cdot)\|_{L^2(\Omega)}\\
  &\leq \frac{2}{n}+2\|u_0\|_{L^2(\Omega_\delta)}+2\sqrt{t_0}C(n),
\end{split}
\end{equation*}
for any $n$ and small $\delta$. By sending $\delta \to 0, t_0 \to 0$, and $n \to \infty$, we get $v(\cdot, t)\to u_0$ as $t\to 0$. The compact argument follows immediately from what we have proved.

\medskip

\noindent\textbf{ Step 2}. In what follows, our goal is to obtain the weak solution of (\ref{EPDE}) with possible effective boundary conditions.

By the preceding compactness argument, $u \to v$ strongly in $C\left([0, T]; L^2(\Omega)\right)$ after passing to a subsequence of $\delta >0$. If $v$ is the weak solution of (\ref{EPDE}) with any boundary condition listed in Table \ref{tb1}, then the existence and uniqueness ensure the convergence without passing a subsequence. The assertion of this theorem now  readily follows.

At the end of the first step, we have already taken care of the initial condition. For simplicity, we take a test function $ \xi \in C^\infty(\overline\Omega\times[0, T])$ with $ \xi=0$ on $\partial\Omega\times(0, T)$ and $ \xi=0$ for $t\in[0,\varepsilon]\cup[T-\varepsilon,T]$ for some small $\varepsilon>0.$ Next, construct a new test function $\overline{\xi}(x,t) $ in the domain $\overline{\Omega}\times [0,T]$ by
\begin{equation*}
   \overline{\xi}(x,t)=\left\{
\begin{array}{ll}
     \xi(x,t), & x\in \overline{\Omega}\backslash \Omega_\delta, \\
        \psi(s,r,t), & x\in\Omega_\delta,
\end{array}
\right.
\end{equation*}
where $ \psi(r,s,t)$ is the solution to the elliptic problem (\ref{AF}), and it is easy to see that $\overline{\xi}\in W^{1,1}_{2,0}(Q_T)$.

According to the definition of the weak solution of (\ref{PDE}), it follows that
\begin{equation*}\label{WKSol}
    \begin{split}
      \mathcal{A}[u,\overline{\xi}]&=k_1\int_{0}^{T}\int_{\Omega_1}\nabla u\cdot \nabla  \xi dxdt+k_2\int_{0}^{T}\int_{\Omega_{d_k}}\nabla u\cdot \nabla  \xi dxdt \\
      & \quad + \int_{0}^{T}\int_{(\Omega\backslash\overline{\Omega}_1)\backslash\overline{\Omega}_{d_k}}\nabla u\cdot A\nabla  \xi dxdt -\int_{Q_T}(u_0 \xi(x,0) +f \xi +u\xi_t )dxdt, \\
    \end{split} \tag{4.1}
\end{equation*}
for the given $d_k > \delta$. Since $u \to v $ weakly in $W^{1, 0}_2(\Omega_1 \times (0, T))$ and $W^{1, 0}_2(\Omega_{d_k} \times (0, T))$, and strongly in $C\left([0, T]; L^2(\Omega)\right)$ as $\delta \to 0$, we summarize
\begin{equation*}\label{convergence}
\left\{
\begin{array}{ll}
     &  \int_{Q_T}u\xi_t dxdt \to \int_{Q_T}v\xi_t dxdt, \\
     & \int_{0}^{T}\int_{\Omega_1}\nabla u\cdot \nabla  \xi dxdt \to \int_{0}^{T}\int_{\Omega_1}\nabla v\cdot \nabla  \xi dxdt,\\
     &\int_{0}^{T}\int_{\Omega_{d_k}}\nabla u\cdot \nabla  \xi dxdt \to \int_{0}^{T}\int_{\Omega_{d_k}}\nabla v\cdot \nabla  \xi dxdt,
\end{array}
\right. \tag{4.2}
\end{equation*}
and 
\begin{equation*}\label{convergence2}
  \begin{split}
    \int_{0}^{T}\int_{(\Omega\backslash\overline{\Omega}_1)\backslash\overline{\Omega}_{d_k}}\nabla u\cdot A\nabla  \xi dxdt=&\int_{0}^{T}\int_{\Omega_{\delta}}\nabla u\cdot A\nabla  \xi dxdt +k_2\int_{0}^{T}\int_{\{x\in \Omega _2 | dist(x, \Gamma_1) <d_k\}}\nabla u\cdot \nabla  \xi dxdt\\
    =&\int_{0}^{T}\int_{\Omega_{\delta}}\nabla u\cdot A\nabla  \xi dxdt +O(\sqrt{d_k})\\
    \to & \lim_{\delta\to 0} \int_{0}^{T}\int_{\Omega_{\delta}}\nabla u\cdot A\nabla  \xi dxdt,
  \end{split}\tag{4.3}
\end{equation*}
in which the convergence results from
\begin{equation*}
  \begin{split}
    &\left|\int_{0}^{T}\int_{\{x\in \Omega _2 | dist(x, \Gamma_1) <d_k\}}\nabla u\cdot \nabla  \xi dxdt \right|\\
\leq &  \left(\int_{0}^{T}\int_{\{x\in \Omega _2 | dist(x, \Gamma_1) <d_k\}} |\nabla u|^2 dxdt \right)^{1/2}\left(\int_{0}^{T}\int_{\{x\in \Omega _2 | dist(x, \Gamma_1) <d_k\}} |\nabla \xi|^2 dxdt \right)^{1/2}\\
    \leq &\left(\int_{0}^{T}\int_{\Omega_2} |\nabla u|^2 dxdt \right)^{1/2}\left(\int_{0}^{T}\int_{ \Omega _2 } |\nabla \xi|^2 dxdt \right)^{1/2}\\
    \leq & O(\sqrt{d_k}).
  \end{split}
\end{equation*}

In addition, by sending $\delta \to 0$ first and then $k\to \infty$, because of (\ref{convergence}) and (\ref{convergence2}), (\ref{WKSol}) gives
\begin{equation*}\label{wksol}
\begin{split}
   \mathcal{L}[v,\xi]=&\int_{0}^{T}\int_{\Omega}\nabla \xi \cdot A_0(x)\nabla v dxdt + \int_{Q_T} \left(v_t \xi - f \xi \right)dxdt = -\lim_{\delta\to 0}\int_{0}^{T}\int_{\Omega_\delta}\nabla \psi\cdot A\nabla  u dxdt.
 \end{split}\tag{4.4}
\end{equation*}

From now on, our problem reduces to investigate the asymptotic behavior of the right-hand side term of (\ref{wksol}). 

Using the curvilinear coordinates $(s,r)$ in (\ref{curvilinear}), in terms of (\ref{derivative}) and (\ref{derivative2}), we thus have
\begin{equation*}\label{RHS}
    \begin{split}
    -\int_{0}^{T}\int_{\Omega_\delta}\nabla \psi\cdot A\nabla  u dxdt =&-\int_0^T\int_{\Gamma_1}\int_0^\delta (\sigma     \psi_r u_r+\mu\nabla_s     \psi \nabla_s u) [1+2H(s)r+\kappa(s)r^2] drdsdt \\
    =&-\int_0^T\int_{\Gamma_1}\int_0^\delta (\sigma     \psi_r u_r+\mu\nabla_{\Gamma_1} \psi \nabla_{\Gamma_1} u)drdsdt \\
     &-\int_0^T\int_{\Gamma_1}\int_0^\delta (\sigma     \psi_r u_r+\mu\nabla_{\Gamma_1} \psi \nabla_{\Gamma_1} u)[2H(s)r+\kappa(s)r^2] drdsdt \\
    & -\int_0^T\int_{\Gamma_1}\int_0^\delta \mu (\nabla_s     \psi \nabla_s u-\nabla_{\Gamma_1} \psi \nabla_{\Gamma_1} u)[1+2H(s)r+\kappa(s)r^2]drdsdt\\
    =:&I+II+III.
     \end{split} \tag{4.5}
\end{equation*}
Thanks to Lemma \ref{est1} and (\ref{eqno34}), it follows from H\"older inequality that
\begin{equation*}\label{sizeII}
\begin{split}
    |II|\leq &\int_0^T \left|\int_{\Gamma_1}\int_0^\delta -(\sigma     \psi_r u_r+\mu\nabla_{\Gamma_1} \psi \nabla_{\Gamma_1} u)[2H(s)r+\kappa(s)r^2] dsdr \right|dt\\
    \leq & O(\delta) \int_0^T \left( \int_{\Gamma_1}\int_0^\delta \sigma \psi_r^2 +\mu(\nabla_{\Gamma_1} \psi )^2 \right)^{1/2}\left( \int_{\Omega} \sigma u_r^2+\mu (\nabla_{\Gamma_1} u)^2 \right)^{1/2}dt\\
    \leq & O(\delta) \int_0^T \frac{1}{\sqrt{t}}\left(\int_{\Gamma_1} \sigma [\psi_r(s,\delta,t)\psi(s,\delta,t) - \psi_r(s,0,t)\psi(s,0,t)]ds \right)^{1/2}dt\\
    \leq & O(\delta) \int_0^T \frac{1}{\sqrt{t}}\left(\int_{\Gamma_1} \sqrt{\sigma\mu} [\Psi_R(s,\delta)\psi(s,\delta,t) - \Psi_R(s,0)\psi(s,0,t)]ds \right)^{1/2}dt.
\end{split}\tag{4.6}
\end{equation*}
Subsequently, applying Lemma \ref{est1} and (\ref{eqno34}) again, we then obtain
\begin{equation*}\label{sizeIII}
\begin{split}
    |III|\leq & \left|\int_0^T\int_{\Gamma_1}\int_0^\delta \mu (\nabla_s     \psi \nabla_s u-\nabla_{\Gamma_1} \psi \nabla_{\Gamma_1} u)[1+2H(s)r+\kappa(s)r^2]drdsdt \right|\\
    \leq & O(\delta) \left|\int_0^T\int_{\Gamma_1}\int_0^\delta \mu  \sum_{ij}g^{ij}_r(s,\overline{r}) \psi_{s_i} u_{s_j}drdsdt \right|\\
    \leq & O(\delta) \int_0^T\int_{\Gamma_1}\int_0^\delta \mu |\sum_{ij} \psi_{s_i} u_{s_j}|drdsdt\\
    \leq & O(\delta) \int_0^T \left( \int_{\Gamma_1}\int_0^\delta \sigma \psi_r^2 +\mu(\nabla_{\Gamma_1} \psi )^2 \right)^{1/2}\left( \int_{\Omega} \sigma u_r^2+\mu (\nabla_{\Gamma_1} u)^2 \right)^{1/2}dt\\
    \leq & O(\delta) \int_0^T \frac{1}{\sqrt{t}}\left(\int_{\Gamma_1} \sigma [\psi_r(s,\delta,t)\psi(s,\delta,t) - \psi_r(s,0,t)\psi(s,0,t)]ds\right)^{1/2} dt\\
    \leq & O(\delta) \int_0^T \frac{1}{\sqrt{t}}\left(\int_{\Gamma_1} \sqrt{\sigma\mu} [\Psi_R(s,\delta)\psi(s,\delta,t) - \Psi_R(s,0)\psi(s,0,t)]ds \right)^{1/2}dt,
\end{split}\tag{4.7}
\end{equation*}
since $g^{ij}(s,r) = g^{ij}(s,0) + r g_r^{ij}(s,r')$ for $r' \in (0, r)$.

In the light of (\ref{size3}), (\ref{sizeII}) and  (\ref{sizeIII}), we are led to
\begin{equation*}\label{s1}
   |II +III| \leq O(T\sqrt{\sigma\mu})^{1/2}\delta, \tag{4.8} 
\end{equation*}
in the case that $ h \to H \in (0,\infty]$ as $\delta \to 0$. For the case that $h \to 0$ as $\delta \to 0$, it is much more complicated. Rather than give a conclusion, we will get into details in the following. 

For ease of use, by recalling (\ref{eq35}), we simplify the expression as 
\begin{equation*}\label{I}
\begin{split}
    I = &-\int_0^T\int_{\Gamma_1}\int_0^\delta (\sigma     \psi_r u_r+\mu\nabla_{\Gamma_1} \psi \nabla_{\Gamma_1} u)drdsdt \\
    =&- \sqrt{\sigma\mu}\int_0^T\int_{\Gamma_1} [\Psi_R(s,h)u(s,\delta,t) -\Psi_R(s,0)u(s,0,t)]dsdt\\
    := & \int_0^T \mathcal{I}dt.
\end{split}\tag{4.9}
\end{equation*}
Let us consider the following cases as $\delta \to 0$
\begin{equation*}
\begin{split}
    &(1) \frac{\sigma}{\delta}\to 0, \quad
    (2) \frac{\sigma}{\delta}\to b\in (0,\infty), \quad (3) \frac{\sigma}{\delta}\to \infty \text{ and } \sigma\delta^3 \to 0,
\end{split}
\end{equation*}
with three subcases $(i) \sigma\mu\to 0$, $(ii)\sqrt{\sigma\mu}\to \gamma\in(0,\infty)$, $(iii) \sigma\mu \to \infty$.

\medskip

Now, we start with the first case that implies $\sigma\to 0$ as $\delta \to 0$.

\smallskip

\noindent Case 1. $\frac{\sigma}{\delta}\to 0$ as $\delta \to 0$.

\noindent Subcase $(1i). \quad  \sigma\mu\to 0$. In this case, combining Lemma \ref{est1} and (\ref{AF}), it is not difficult to verify that
\begin{equation*}\label{eq410}
\begin{split}
    &\left|\int_{0}^{T}\int_{\Omega_\delta}\nabla \xi\cdot A(x)\nabla  u dxdt\right| \\
\leq &O(1) \int_0^T \left( \int_{\Gamma_1}\int_0^\delta \sigma \psi_r^2 +\mu |\nabla_{\Gamma_1} \psi |^2 drds \right)^{1/2}\left( \int_{\Omega} \sigma u_r^2+\mu |\nabla_{\Gamma_1} u|^2 drds\right)^{1/2}dt\\
    \leq & max \{O((\sigma\mu)^{1/4}), O(\sigma)\} \\  \to & 0 \text{ as } \delta\to 0,
\end{split}\tag{4.10}
\end{equation*}
with the assistance of (\ref{eqno36}), (\ref{size1}), (\ref{size11}), and (\ref{size3}). Consequently, as $\delta \to 0$, (\ref{wksol}) is equal to
\begin{equation*}
    \mathcal{L}[v,\xi] = 0.
\end{equation*}
We thus infer that $v$ satisfies the boundary condition on $\Gamma_1 \times (0, T)$ 
\begin{equation*}
    k_1\frac{\partial v_1}{\partial\textbf{n}} =  k_2\frac{\partial v_2}{\partial\textbf{n}}.
\end{equation*}

Moreover,  we intend to determine other effective boundary conditions of $v_1$ on $\Gamma_1$.
Take the test function $\xi$ that depends on $\delta$, with $\xi \equiv 0$ in $\overline{\Omega}_2$, and $\psi$ is defined similarly in (\ref{AF}) only by changing $g_2(s) \equiv 0$. Since $u \to v $ weakly in $W^{1, 0}_2(\Omega_1 \times (0, T))$, as $\delta \to 0$, it follows from (\ref{wksol}) that \begin{equation*}\label{eqv_1}
\begin{split}
   \mathcal{L}[v_1,\xi]=&\int_{0}^{T} \int_{\Omega_1} (v_1)_t \xi dxdt + \int_{0}^{T}\int_{\Omega_1}\nabla v_1 \cdot A_0\nabla  \xi dxdt-\int_{0}^{T} \int_{\Omega_1}f \xi dxdt\\
   =& -\lim_{\delta\to 0}\int_{0}^{T}\int_{\Omega_\delta}\nabla \psi\cdot A\nabla  u dxdt\\
   \leq &O(1) \int_0^T \left( \int_{\Gamma_1}\int_0^\delta \sigma \psi_r^2 +\mu |\nabla_{\Gamma_1} \psi |^2 drds \right)^{1/2}\left( \int_{\Omega} \sigma u_r^2+\mu |\nabla_{\Gamma_1} u|^2 drds\right)^{1/2}dt\\
    \leq & max \{O((\sigma\mu)^{1/4}), O(\frac{\sigma}{\delta})\},
 \end{split}\tag{4.11}
\end{equation*}
where we have used (\ref{size1}) and (\ref{s1}). So, $\mathcal{L}[v_1, \xi] = 0$, resulting in the boundary condition on $\Gamma_1 \times (0, T)$ 
\begin{equation*}
    k_1\frac{\partial v_1}{\partial\textbf{n}} = 0.
\end{equation*}

\smallskip

\noindent Subcase $(1ii)$. $\sqrt{\sigma\mu}\to \gamma \in (0,\infty)$. In this case, $h \to \infty$ and $\mu\delta \to \infty$ as $\delta \to 0$. It is clear to note that
\begin{equation*}
\begin{split}
    \mathcal{I} =&- \sqrt{\sigma\mu}\int_{\Gamma_1} [\Psi_R(s,h)u(s,\delta,t) -\Psi_R(s,0)u(s,0,t)]ds
    \to  \gamma \int_{\Gamma_1} (v_1+v_2)\mathcal{J}^\infty[g_1]ds,
\end{split}
\end{equation*}
as $\delta \to 0$ because of the weak convergence of $u \to v$ in $W^{1, 0}_2(\Omega_1 \times (0, T))$. 

A combination of (\ref{eq35}), (\ref{wksol}), (\ref{RHS}) and (\ref{s1}) yields
\begin{equation*}
\begin{split}
    \mathcal{L}[u,\xi] =& - \sqrt{\sigma\mu}\int_{\Gamma_1} [\Psi_R(s,h)u(s,\delta,t) -\Psi_R(s,0)u(s,0,t)]ds +O(\delta)
    \to  \gamma \int_{\Gamma_1} (v_1+v_2)\mathcal{J}^\infty[g_1]ds,
\end{split}
\end{equation*}
as $\delta \to 0$, which implies that $v$ satisfies the boundary condition on $\Gamma_1 \times (0, T)$
\begin{equation*}
  k_1\frac{\partial v_1}{\partial \textbf{n}} - k_2\frac{\partial v_2}{\partial \textbf{n}} = \gamma \mathcal{J}^{\infty}[v_1+v_2].  
\end{equation*}

To obtain the effective boundary conditions of $v_1$ on $\Gamma_1$, the remainder of this procedure is analogous to that in Subcase $(1i)$. By choosing the test function $\xi$ such that $\xi = 0$ in $\overline{\Omega}_2$, say, $g_2(s) \equiv 0$, on account of  (\ref{eq35}), (\ref{sizeII}), (\ref{sizeIII}), and (\ref{s1}), we produce
\begin{equation*}
\begin{split}
   \mathcal{L}[v_1,\xi]=&\int_{0}^{T} \int_{\Omega_1} (v_1)_t \xi dxdt + \int_{0}^{T}\int_{\Omega_1}\nabla v_1 \cdot A_0\nabla  \xi dxdt-\int_{0}^{T} \int_{\Omega_1}f \xi dxdt\\
   =& - \sqrt{\sigma\mu}\int_{\Gamma_1} [\Psi_R(s,h)u(s,\delta,t) -\Psi_R(s,0)u(s,0,t)]ds +O(\delta)\\
   \to & \gamma \int_{\Gamma_1} v_1\mathcal{J}_1^\infty[g_1]ds.
 \end{split}\tag{4.12}
\end{equation*}
Thus, due to the arbitrariness of the test function, $v_1$ satisfies the boundary condition on $\Gamma_1 \times (0, T)$
\begin{equation*}
  k_1\frac{\partial v_1}{\partial \textbf{n}}  = \gamma \mathcal{J}_1^{\infty}[v_1].
\end{equation*}
Immediately, we have $k_2\frac{\partial v_2}{\partial \textbf{n}}  = -\gamma \mathcal{J}^{\infty}[v_2] $ on $\Gamma_1 \times (0, T)$ since $\mathcal{J}_1^{\infty} = \mathcal{J}^{\infty}$.

\medskip

\noindent Subcase $(1iii)$. $\sigma\mu \to \infty$. In this case, $h \to \infty$ and $\mu\delta \to \infty$ as $\delta \to 0$. Then, divided by $\sqrt{\sigma\mu}$ on both sides of (\ref{wksol})  and sending $\delta \to 0$, in view of $ (\ref{wksol})-(\ref{s1})$, we obtain
\begin{equation*}
\begin{split}
    \int_{\Gamma_1} (v_1+v_2)\mathcal{J}^\infty[g_1]ds = 0,
\end{split}
\end{equation*}
from which we have  $\nabla_{\Gamma_1}(v_1+v_2)=0$ because the range of $\mathcal{J}^\infty[\cdot]$ contains eigenfunctions of $-\Delta_{\Gamma_1} $. 

The effective boundary conditions of $v_1$ on $\Gamma_1$ follow by taking the test function $\xi$ with $\xi = 0$ in $\overline{\Omega}_2$. Then, proceeding as in the above analysis, due to  $\mathcal{J}_2^\infty[g_1] = 0$, we get
\begin{equation*}
\begin{split}
    \int_{\Gamma_1} v_1\mathcal{J}_1^\infty[g_1]ds = 0,
\end{split}
\end{equation*}
meaning that $\nabla_{\Gamma_1} v_1 = \nabla_{\Gamma_1} v_2 = 0$ on $\Gamma_1$. This boundary condition indicates that $v$ is a constant on $\Gamma_1$ in the spatial variable, but it can be a function of $t$.
 
Assume further that $\xi=m(t)$ for some smooth function $m(t)$ on $\Gamma_1$, and $\psi(s,r,t) = m(t) (\delta - r)/\delta$. Carrying out a direct computation, we arrive at 
\begin{equation*}\label{eqno413}
\begin{split}
   \mathcal{L}[v_1,\xi]=&-\lim_{\delta\to 0}\int_{0}^{T}\int_{\Omega_\delta}\nabla \psi\cdot A\nabla  u dxdt\\
    =&\lim_{\delta\to 0}\int_{0}^{T}\int_{\Gamma_1}\frac{\sigma}{\delta} m(t) \int_0^\delta u_r (1+2H(s)r + \kappa(s)r^2) drdsdt\\
   =&\lim_{\delta\to 0}\int_{0}^{T}\int_{\Gamma_1}\frac{\sigma}{\delta} m(t) \left(u(s, \delta^-, t) - u(s, 0^+, t) \right)dsdt\\
   = &0,
 \end{split}\tag{4.13}
\end{equation*}
from which $v_1$ satisfies $ \int_{\Gamma_1} k_1\frac{\partial v_1}{\partial\textbf{n}} ds = 0$ on $\Gamma_1 \times (0, T)$.

\smallskip

Eventually, consider the test function $\xi$ with $\xi(s,   0^+, t) = m(t)$ being a constant in the spatial variable on $\Gamma_1$. By transmission conditions in (\ref{trans}), it turns out that
\begin{equation*}\label{414}
\begin{split}
    \int_{0}^{T}\int_{\Omega_\delta}A\nabla u \cdot \nabla  \xi dxdt
  =&-\int_{0}^{T}\int_{\Omega_\delta} \nabla \cdot (A \nabla u) \xi  dxdt+\int_{0}^{T}\int_{\partial\Omega_\delta}\xi  A \nabla u \cdot \textbf{n}  dsdt\\
  =&-\int_{0}^{T}\int_{\Omega_\delta} \nabla \cdot (A\nabla u) ( \xi- m(t)) dxdt\\
  &-\int_{0}^{T}\int_{\Omega_\delta}  \nabla \cdot (A \nabla u)  m(t) dxdt+\int_{0}^{T}\int_{\partial\Omega_\delta}\xi  A \nabla u \cdot \textbf{n}  dsdt \\
 =&\int_{0}^{T}\int_{\Omega_\delta} (u_t - f) ( \xi- m(t)) dxdt + \int_{0}^{T}\int_{\partial\Omega_\delta}(\xi - m(t) ) A \nabla u \cdot \textbf{n}  dsdt \\
\leq & \left|\int_{\varepsilon}^{T}\int_{\Omega_\delta} ( \xi- m(t)) (u_t-f) dxdt  \right| + \sigma\int_{0}^{T}\int_{\Gamma_2}( \xi- m(t)) \frac{\partial u}{ \partial \textbf{n}} dsdt\\
  \leq  & \left(\int_{\varepsilon}^{T}\int_{\Omega_\delta} ( \xi- m(t))^2dxdt  \right)^\frac{1}{2}\left(\int_{\varepsilon}^{T}\int_{\Omega_\delta} (u_t-f)^2 dxdt  \right)^\frac{1}{2}\\
   &\quad +\sigma\int_{0}^{T}\int_{\Gamma_1}\left( \xi(s, \delta^-, t)- m(t) \right) u_r(s,\delta^-,t)[1+2H\delta+\kappa \delta^2] dsdt\\
  \leq  & O(\delta^{1/2})+k_2\int_{\varepsilon}^{T}\int_{\Gamma_1}\left( \xi (s, \delta^-, t) - m(t)\right) u_r(s,\delta^+,t)[1+2H\delta+\kappa \delta^2] dsdt\\
\leq & O(\delta^{1/2})+O(\delta)\left( \int_{\varepsilon}^{T}\int_{\Gamma_1}u_r^2(s,\delta^+,t) dsdt \right)^{1/2}\\
\leq & O(\delta^{1/2})+O(\delta)\left( \int_{\varepsilon}^{T}\int_{\Omega_2} |D^2u|^2 dxdt \right)^{1/2},
\end{split}\tag{4.14}
\end{equation*}
where PDE (\ref{PDE}) was used, and the trace theorem holds from Theorem \ref{thm2}.

Consequently, we have $\mathcal{L}[v,\xi]=0$, leading to the boundary condition on $\Gamma_1 \times (0, T)$ 
$$
\int_{\Gamma_1} \left(k_1\frac{\partial v_1}{\partial\textbf{n}} - k_2\frac{\partial v_2}{\partial\textbf{n}}  \right) ds = 0.
$$

\bigskip

\noindent Case 2. $\frac{\sigma}{\delta}\to b \in (0, \infty)$ as $\delta \to 0$. Thus, $\sigma \to 0$ as $\delta \to 0$.

\noindent Subcase $(2i)$. $\sigma\mu\to 0$. In this case, $h \to 0$ and $\mu\delta \to 0$ as $\delta \to 0$. Due to (\ref{eq410}),  we obtain $\mathcal{L}[v,\xi] = 0$, from which $v$ satisfies the boundary condition on $\Gamma_1 \times (0, T)$ 
\begin{equation*}
    k_1\frac{\partial v_1}{\partial\textbf{n}} =  k_2\frac{\partial v_2}{\partial\textbf{n}}.
\end{equation*}

To derive effective boundary conditions for $v_1$ on $\Gamma_1$, we take the test function $\xi$ with $\xi \equiv 0$ in $\overline{\Omega}_2$ again. Thus, (\ref{size2}) gives
\begin{equation*}\label{eq414}
\begin{split}
    \mathcal{I} = &- \sqrt{\sigma\mu} \int_{\Gamma_1} \left(\Psi_R(s, h) -\Psi_R(s, 0) \right) u(s,\delta,t)ds -\sqrt{\sigma\mu}\int_{\Gamma_1}\Psi_R(s,0)\left(u(s,\delta,t)- u(s,0,t) \right)ds\\
    = &  \mu\delta\int_{\Gamma_1} \left( \frac{1}{2} \Delta_{\Gamma_1} (g_1 +g_2) + O(h) \right) u(s, \delta, t)ds\\
    &  - \int_{\Gamma_1}\left[\frac{\sigma}{\delta}\left((g_2 - g_1)+O( h^2)\right)\right] \left[u(s, \delta, t)- u(s, 0, t)\right]ds,
\end{split}\tag{4.15}
\end{equation*}
from which 
$$I \to   b \int_{\Gamma_1} g_1(v_2- v_1)ds \quad  \text{  as }  \quad \delta \to 0.
$$
Furthermore, (\ref{sizeII}) and (\ref{sizeIII}) yield
\begin{equation*}\label{eq415}
\begin{split}
    | II +III | \leq & O(\delta) \int_0^T \frac{1}{\sqrt{t}}\left(\int_{\Gamma_1} \sqrt{\sigma\mu} [\Psi_R(s,\delta)\psi(s,\delta,t) - \Psi_R(s,0)\psi(s,0,t)]ds \right)^{1/2}dt\\
    \to & 0 \quad  \text{  as }  \quad \delta \to 0.
\end{split}\tag{4.16}
\end{equation*}

Combining (\ref{eq414}) and (\ref{eq415}), we get
$$
\mathcal{L}[v_1, \xi] = b \int_0^T \int_{\Gamma_1} g_1(v_2- v_1)ds dt.
$$
Because of the arbitrariness of $\xi$, $v_1$ satisfies the boundary condition on $\Gamma_1 \times (0, T)$
\begin{equation*}
    k_1\frac{\partial v_1}{\partial\textbf{n}} =   b   (v_2- v_1).
\end{equation*}

\medskip

\noindent Subcase $(2ii)$. $\sqrt{\sigma\mu}\to \gamma \in (0,\infty)$. In this case, $\mu\delta \to \beta \in (0, \infty)$ and $h \to H= \beta /\gamma (\gamma / b) \in (0, \infty)$ as $\delta \to 0$.

It is a consequence of the weak convergence of $u$ that as $\delta \to 0$,
\begin{equation*}
\begin{split}
    \mathcal{I} =&- \sqrt{\sigma\mu}\int_{\Gamma_1} [\Psi_R(s,h)u(s,\delta,t) -\Psi_R(s,0)u(s,0,t)]ds
    \to  \gamma \int_{\Gamma_1} (v_1+v_2)\mathcal{J}^{\beta/\gamma}[g_1]ds.
\end{split}
\end{equation*}
Thanks to (\ref{eq35}) and (\ref{s1}), it follows from (\ref{wksol}) and (\ref{RHS}) that
\begin{equation*}\label{eqno417}
\begin{split}
    \mathcal{L}[u,\xi] =& - \sqrt{\sigma\mu}\int_{\Gamma_1} [\Psi_R(s,h)u(s,\delta,t) -\Psi_R(s,0)u(s,0,t)]ds +O(\delta)\\
    \to & \gamma \int_{\Gamma_1} (v_1+v_2)\mathcal{J}^{\beta/\gamma}[g_1]ds,
\end{split}\tag{4.17}
\end{equation*}
which suggests the boundary condition on $\Gamma_1 \times (0, T)$ 
\begin{equation*}\label{eqs418}
   k_1\frac{\partial v_1}{\partial\textbf{n}} - k_2\frac{\partial v_2}{\partial\textbf{n}} = \mathcal{J}^{\beta/\gamma}[v_1 + v_2]. \tag{4.18}
\end{equation*}

On the other hand, we understand the boundary condition of $v_1$ on $\Gamma_1$ by revising $\xi$ with $\xi \equiv 0$ in $\overline{\Omega}_2$. Comparably, employing the analysis similar to the one used in (\ref{eqno417}), we have 
\begin{equation*}
\begin{split}
   & \mathcal{L}[u_1,\xi] = - \sqrt{\sigma\mu}\int_{\Gamma_1} [\Psi_R(s,h)u(s,\delta,t) -\Psi_R(s,0)u(s,0,t)]ds +O(\delta)\\
    \to & \mathcal{L}[v_1,\xi] = \gamma \int_{\Gamma_1} \left( \mathcal{J}_1^{\beta/\gamma}[v_1]g_1 - \mathcal{J}_2^{\beta/\gamma}[v_1]g_2 \right) ds,
\end{split}
\end{equation*}
where we made use of the fact $\mathcal{J}_1^{\beta/\gamma}[\cdot]$ and $\mathcal{J}_2^{\beta/\gamma}[\cdot]$ are linear and symmetric. Thus, $v_1$ satisfies the boundary condition on $\Gamma_1 \times (0, T)$
\begin{equation*}\label{ebc1}
\begin{split}
   k_1\frac{\partial v_1}{\partial\textbf{n}}
    = \mathcal{J}_1^{\beta/\gamma}[v_1] - \mathcal{J}_2^{\beta/\gamma}[v_2].
\end{split}\tag{4.19}
\end{equation*}
Combining this with (\ref{eqs418}), we obtain the boundary condition of $v_2$ on $\Gamma_1$ as well
\begin{equation*}\label{ebc2}
\begin{split}
   k_2\frac{\partial v_2}{\partial\textbf{n}}
    = \mathcal{J}_2^{\beta/\gamma}[v_1] - \mathcal{J}_1^{\beta/\gamma}[v_2].
\end{split}\tag{4.20}
\end{equation*}

\medskip

\noindent Subcase $(2iii)$. $\sigma\mu \to \infty$. In this case, $h \to \infty$ and $\mu\delta \to \infty$ as $\delta \to 0$. By taking the test function $\xi $ satisfying $\xi \equiv 0$ on $\overline{\Omega}_2$, we are led to 
\begin{equation*}
\begin{split}
    \int_{\Gamma_1} v_1\mathcal{J}^\infty[g_1]ds = 0.
\end{split}
\end{equation*}
This represents that $\nabla_{\Gamma_1} v_1 = 0$ on $\Gamma_1 \times (0, T)$. Next, assume further that $\xi(s, 0^+, t) = m(t)$ on $\Gamma_1$ and $\psi(s,r,t) = m(t) \frac{\delta-r}{\delta}$. Then, implementing the same analysis in (\ref{eqno413}) gives 
$$\int_{\Gamma_1} \left( k_1\frac{\partial v_1}{\partial \textbf{n}} - b (v_2-v_1) \right) ds=0 \quad \text{ on } \quad  \Gamma_1 \times (0, T).
$$

\smallskip

Going back to the standard test function $\xi$ that is independent of $\delta$ and divided by $\sqrt{\sigma\mu}$ on both sides of (\ref{wksol}), by sending $\delta \to 0$, we then get
\begin{equation*}
\begin{split}
    \int_{\Gamma_1} (v_1 + v_2)\mathcal{J}^\infty[g_1]ds = 0.
\end{split}
\end{equation*}
Thus, it is apparent that $\nabla_{\Gamma_1} v_2 = 0$ on $\Gamma_1$. By further assumption that $\xi(s, 0^+, t) = m(t)$, it follows from (\ref{414}) that  $\mathcal{L}[v, \xi] = 0$, implying that $v$ satisfies the boundary condition on $\Gamma_1 \times (0, T)$
$$
\int_{\Gamma_1} \left( k_1\frac{\partial v_1}{\partial \textbf{n}} -  k_2\frac{\partial v_2}{\partial \textbf{n}} \right) ds=0.
$$

\bigskip

\noindent Case 3.  $\frac{\sigma}{\delta}\to \infty$ and $\sigma\delta^3 \to 0$ as $\delta \to 0$.

\smallskip

Before diving into several subcases, we conclude with $v_1$ and $v_2$.  By the weak convergence of $u \to v$ in $W_2^{1,0}(Q_T)$ as $\delta \to 0$, it follows that
\begin{equation*}
    \begin{split}
        \int_0^T\frac{\sigma}{\delta}\int_{\Gamma_1} (u(s,\delta,t) - u(s,0,t)) \xi(s,0,t) ds dt 
        = &\frac{\sigma}{\delta}\int_0^T \int_{\Gamma_1} \int_0^\delta u_r \xi(s,0,t) dr dsdt\\
        \leq &O\left(\sqrt{\frac{\sigma}{\delta}} \right) \left(\int_0^T\int_\Omega \nabla u \cdot A \nabla udxdt\right)^{1/2}\\
        \leq &O\left(\sqrt{\frac{\sigma}{\delta}} \right),
    \end{split}\tag{4.21}
\end{equation*}
in which we used H\"older inequality and Lemma \ref{est1}. Then divided by $\sqrt{\sigma/\delta}$ on both sides and sending $\delta \to 0$, we derive
\begin{equation*}
    \begin{split}
        &\int_0^T \int_{\Gamma_1} (v_2 - v_1) \xi(s,0,t) ds dt =0,
    \end{split}
\end{equation*}
which means that  $v_1=v_2$ on $\Gamma_1$ due to the arbitrariness of the test function.

\medskip

\noindent Subcase $(3i)(3ii). \sqrt{\sigma\mu} \to \gamma \in [0,\infty)$. In this case, $h\to 0$ and $\mu\delta \to 0$ as $\delta \to 0$. Moreover, if $\mu/\sigma \to c \in (0, \infty ]$ as $\delta \to 0$, then $\sigma\delta \to 0$ as $\delta \to 0$. Since
\begin{equation*}
    \begin{split}
        \sigma \int_0^T \int_{\Gamma_1}\frac{g_2-g_1}{\delta} \left[u(s, \delta,t)-u(s, 0,t)\right] ds dt
        =&\sigma \delta \int_0^T\int_{\Gamma_1} \frac{g_2-g_1}{\delta}\frac{1}{\delta}\left(\int_0^\delta u_r dr \right) ds dt
        \leq  O(\sqrt{\sigma \delta}),
    \end{split}
\end{equation*}
from (\ref{eq414}), we get $I \to 0$ as $\delta \to 0$. Because (\ref{eq415}) also holds in this case, we obtain $\mathcal{L}[v, \xi] = 0$, where $v$ satisfies the boundary condition on $\Gamma_1 \times (0, T)$
\begin{equation*}
k_1\frac{\partial v_1}{\partial\textbf{n}} = k_2\frac{\partial v_2}{\partial\textbf{n}}.
\end{equation*}

If $\mu/\sigma \to 0$ as $\delta \to 0$, then implementing the integration by parts and H\"older inequality, we get
\begin{equation*}\label{eqno418}
    \begin{split}
        \left|\sigma \int_0^\delta u_r dr
        -k_2 u_r(s,\delta^+,t)\delta\right|
        =& \sigma \left|\int_0^\delta r u_{rr} dr\right| \leq  O \left(\sigma^{1/2} \delta^{3/2} \right) \left(  \int_0^\delta \sigma u_{rr}^2 dr\right)^{1/2},
    \end{split} \tag{4.22}
\end{equation*}
where the transmission conditions (\ref{trans}) were used. Observe
\begin{equation*}
  \begin{split}
       \left(\sigma \int_0^\delta u_r dr \right)^2 
      &\leq   2\left(  \left| k_2 u_r(s, \delta^+,t)\delta \right|^2 +O(\sigma \delta^3)\int_0^\delta \sigma u_{rr}^2 dr \right),\\
  \end{split}
\end{equation*}
from which it follows that
\begin{equation*}\label{eqno423}
    \begin{split}
        \sigma \int_0^T \int_{\Gamma_1}\frac{g_2-g_1}{\delta} \left[u(s, \delta,t)-u(s, 0,t)\right] ds dt
        =&\sigma \delta \int_0^T\int_{\Gamma_1} \frac{g_2-g_1}{\delta}\frac{1}{\delta}\left(\int_0^\delta u_r dr \right) ds dt\\
        \leq& O(1)  \left(\int_0^T\int_{\Gamma_1} \left(\sigma\int_0^\delta u_r dr\right)^2  dsdt\right)^{1/2}\\
        \leq & O(\delta^{1/2}) \left(\int_0^T \int_{\Gamma_1} \delta  \left| k_2 u_r(s, \delta^+,t) \right|^2 +(\sigma\delta^2 + \delta^2 + \mu\delta^2) \right)^{1/2}\\
        \leq &O(1) \left( \delta^2 \int_{0}^{T}\int_{\Omega_2} |D^2u|^2 dxdt + (\sigma\delta^3 + \delta^3 + \mu\delta^3)\right)^{1/2},
    \end{split}\tag{4.23}
\end{equation*}
with the help of the trace theorem and Lemma \ref{lemm2.21}. 

By (\ref{eqno423}) and the assumption that $\sigma\delta^3 \to 0$, (\ref{eq414}) and (\ref{eq415}) give $\mathcal{L}[v,\xi] = 0$.
Thus, $v$ satisfies the boundary condition on $\Gamma_1 \times (0, T)$
$$
k_1\frac{\partial v_1}{\partial\textbf{n}} = k_2\frac{\partial v_2}{\partial\textbf{n}}.
$$

\medskip

\noindent Subcase $(3iii)$. $\sigma\mu \to \infty$. If $h \to H \in (0, \infty]$, then divided by $\sqrt{\sigma\mu}$ on both sides of (\ref{wksol})  and sending $\delta \to 0$, by (\ref{s1}) and (\ref{I}), we have
\begin{equation*}
\begin{split}
    \int_{\Gamma_1} v \mathcal{J}^{H}[g_1] ds = 0.
\end{split}
\end{equation*}
Using the analogous method as in Subcase $(1iii)$, we find that $ \nabla_{\Gamma_1} v = 0$ and
$$
\int_{\Gamma_1} \left(k_1\frac{\partial v_1}{\partial\textbf{n}} - k_2\frac{\partial v_2}{\partial\textbf{n}} \right) ds = 0.
$$

On the other hand, if $h \to 0$ as $\delta \to 0$, then we first consider the case that $\frac{\mu}{\sigma} \to c \in (0,\infty]$ as $\delta \to 0$. For the case that $\mu\delta \to \infty$, divided by $\sigma\delta$ on both sides of (\ref{wksol}) and sending $\delta\to 0$, from (\ref{eq414}) and (\ref{eq415}), we find that 
\begin{equation*}
    \int_0^T\int_{\Gamma_1}v(s,0,t)\Delta_{\Gamma_1}  \xi(s,0,t)dsdt=0,
\end{equation*}
revealing $ \nabla_{\Gamma_1} v = 0$. Thus, by the same method used in Subcase $(1iii)$, we have
$$
\int_{\Gamma_1} \left(k_1\frac{\partial v_1}{\partial\textbf{n}} - k_2\frac{\partial v_2}{\partial\textbf{n}} \right) ds = 0.
$$

For the case that $\mu\delta \to \beta \in [0, \infty)$, a combination of  (\ref{eq414}), (\ref{eq415}) and (\ref{eqno423}) leads to
\begin{equation*}\label{eqno424}
   \mathcal{L}[v,\xi] = \beta \int_{\Gamma_1} v \Delta_{\Gamma_1} g_1 ds,\tag{4.24}
\end{equation*}
in which we are left to prove $v\in L^2\left((0,T); H^1(\Gamma_1)\right)$ by the weak solution in Definition \ref{def2}, showing that $v$ is a weak solution of (\ref{EPDE}) together with the boundary condition on $\Gamma_1 \times (0, T)$
\begin{equation*}\label{eqno425}
     k_1\frac{\partial v_1}{\partial \textbf{n}}-k_2\frac{\partial v_2}{\partial \textbf{n}}= \beta \Delta_{\Gamma_1} v.  \tag{4.25}
\end{equation*}

In the following, assume $\overline{v}$ is the weak solution of (\ref{EPDE}) with boundary condition (\ref{eqno425}) that also satisfies (\ref{eqno424}). Then, it suffices to show $v=\overline{v}$. 

Let us consider the problem (\ref{EPDE}) for $w = v-\overline{v}$ since $v$ belongs to all spaces in the first step by the compactness argument. Therefore, $ w$ is the weak solution of (\ref{EPDE}) with $u_0=f=0$. 

Fix $t_1\in (t_0,T)$ for any small $t_0\in (0,T)$,  (\ref{eqno424}) is transformed into
\begin{equation*}\label{eq416}
\begin{split}
    \int_{t_0}^{t_1} \int_{\Omega}(w_t\xi+A_0\nabla w \cdot \nabla \xi) dxdt&=\beta\int_{t_0}^{t_1} \int_{\Gamma_1} w(s,0,t) \Delta_{\Gamma_1} \xi(s,0,t) dsdt.\\
\end{split}\tag{4.26}
\end{equation*}
Then, choose the test function $\xi \in C^2(\overline{\Omega} \times [0,T])$ in the curvilinear coordinates $(s,r)$
\begin{equation*}
   \xi(s,r,t)=\left\{
\begin{array}{cc}
     \omega(s,t) \eta(r), & -2\epsilon \leq r \leq 2 \epsilon, \\
        0, & \text{otherwise},
\end{array}
\right.
\end{equation*}
such that $\eta$ is a smooth cut-off function in the $r$ variable with $0\leq \eta\leq 1$, $\eta=1$ for $|r| \leq \epsilon$, and $\eta=0$ for $ |r| \geq 2\epsilon$, where $\epsilon$ is small; $\omega(s,t) \in C^2(\Gamma_1 \times [0,T])$. 

From (\ref{eq416}), we are led to
\begin{equation*}\label{eq417}
    \begin{split}
        &\beta \left|\int_{t_0}^{t_1} \int_{\Gamma_1} w\Delta_{\Gamma_1} \xi dsdt \right|\\
        =&\left|\int_{t_0}^{t_1} \int_{\Omega}(w_t\xi+ A_0\nabla w  \cdot \nabla \xi) dxdt\right |\\
        \leq & C \left(\| w \|_{W^{1,1}_2(\Omega_1 \times (t_0,t_1))} + \| w \|_{W^{1,1}_2(\Omega\backslash\overline{\Omega}_1 \times (t_0,t_1))} \right)\| \omega\|_{L^2((t_0,t_1);H^1(\Gamma_1))}\\
        \leq &O(1),
    \end{split}\tag{4.27}
\end{equation*}
which results from that $w \in W^{1,1}_2(\Omega_1 \times (t_0,t_1))$ and  $W^{1,1}_2(\Omega\backslash\overline{\Omega_1} \times (t_0,t_1)) $.

Furthermore, considering $\omega$ with
\begin{equation*}\label{eq418}
    \begin{split}
        \int_{t_0}^{t_1} \int_{\Gamma_1} \omega dsdt &=0,
    \end{split}\tag{4.28}
\end{equation*}
we define the linear functional on that $\omega$ as
\begin{equation*}
    \begin{split}
       \omega \to  \int_{t_0}^{t_1} \int_{\Gamma_1} w \Delta_{\Gamma_1} \omega dsdt,
    \end{split}
\end{equation*}
which is well-defined by $(\ref{eq417})$ as well. Then, this functional can be extended to the Hilbert space
\begin{equation*}
    \begin{split}
       \mathbb{H} = \{ \omega \in L^2\left((t_0,t_1);H^1(\Gamma_1)\right):  \int_{t_0}^{t_1}\int_{\Gamma_1}\omega  dsdt &=0\}
    \end{split}
\end{equation*}
with the inner product as
\begin{equation*}
    \begin{split}
       -\int_{t_0}^{t_1} \int_{\Gamma_1} \nabla_{\Gamma_1} \omega_1 \cdot \nabla_{\Gamma_1} \omega_2 dsdt.
    \end{split}
\end{equation*}
By Riesz representation theorem, there is some $z \in \mathbb{H} $ satisfying
\begin{equation*}\label{eq429}
    \begin{split}
       -\int_{t_0}^{t_1} \int_{\Gamma_1}  \nabla_{\Gamma_1} z \cdot  \nabla_{\Gamma_1} \omega dsdt=& \int_{t_0}^{t_1} \int_{\Gamma_1} w \Delta_{\Gamma_1} \omega dsdt\\ =&\int_{t_0}^{t_1} \int_{\Gamma_1} z \Delta_{\Gamma_1} \omega dsdt.
    \end{split}\tag{4.29}
\end{equation*}
Eventually, we infer from (\ref{eq429}) that
\begin{equation*}\label{eq430}
    \begin{split}
       \int_{t_0}^{t_1} \int_{\Gamma_1} (w-z) \Delta_{\Gamma_1} \omega dsdt= 0. 
    \end{split} \tag{4.30}
\end{equation*}

By Riesz theorem, (\ref{eq430}) indicates that $w - z = m(t)$ for some function $m(t) \in \mathbb{H}$ and thus $ w \in L^2\left((0, T);H^1(\Gamma_1)\right)$. Moreover, employing integration by parts, it follows from (\ref{eq416}) that 
\begin{equation*}
    \begin{split}
       \int_{\Omega} w^2 (x,t_1) dxdt \leq \int_{\Omega} w^2 (x,t_0) dxdt.
    \end{split} 
\end{equation*}
Sending $t_0 \to 0$, the assertion (\ref{eqno425}) follows.

\medskip

We now move on to the remaining case where  $\frac{\mu}{\sigma} \to 0$ with the assumption that $\sigma\delta^3 \to 0$ as $\delta \to 0$. In this case, $h \to 0$ as $\delta \to 0$. 

If $\mu\delta \to \beta \in [0, \infty)$ as $\delta \to 0$, then (\ref{eqno424}) follows immediately from  (\ref{eq414}), (\ref{eq415}), and (\ref{eqno423}). Consequently, the method used above also applies to this case, giving rise to the boundary condition on $\Gamma_1 \times (0, T)$
\begin{equation*}
     k_1\frac{\partial v_1}{\partial \textbf{n}}-k_2\frac{\partial v_2}{\partial \textbf{n}}= \beta \Delta_{\Gamma_1} v. 
\end{equation*}

\medskip

Finally, we treat the case that  $\mu\delta \to \infty$ as $\delta \to 0$. If $h \to 0$ as $\delta \to 0$, then divided  by $\mu \delta$ on both sides of (\ref{wksol}) and sending $\delta \to 0$, in view of (\ref{eq414}) and (\ref{eq415}), we have 
\begin{equation*}
\begin{split}
    \int_{\Gamma_1} v \Delta_{\Gamma_1} g_1 ds = 0.
\end{split}
\end{equation*}

This condition  suggests that $\nabla_{\Gamma_1} v = 0$ on $\Gamma_1 \times (0,T)$ since the range of $-\Delta_{\Gamma_1}$ contains its eigenfunctions. As before, assume further that $\xi(s, 0, t) = m(t)$. According to (\ref{414}), we have  $\mathcal{L}[v, \xi] = 0$, meaning that $v$ satisfies the boundary condition on $\Gamma_1 \times (0, T)$
$$
\int_{\Gamma_1} \left( k_1\frac{\partial v_1}{\partial \textbf{n}} -  k_2\frac{\partial v_2}{\partial \textbf{n}} \right) ds=0.
$$

\smallskip

In conclusion, this completes the proof of the theorem by what we have already proven.
\end{proof}


\section*{Acknowledgments} The author is grateful to his advisor Professor Xuefeng Wang for his patient discussions. The author also thanks  the anonymous referees for their helpful comments and suggestions.


%
%
%
%
%
%
%
%
%
%
%
%
\

\end{document}